\documentclass[a4paper,11pt]{amsart}
\usepackage[utf8]{inputenc}
\usepackage[english]{babel}
\usepackage{indentfirst}
\usepackage{csquotes}
\usepackage{soul} 
\sethlcolor{green} 
\usepackage{enumerate}

\usepackage{amsmath,amsthm,amssymb,mathrsfs,thmtools}
\usepackage{dsfont}

\usepackage{enumitem}
\usepackage[toc,page]{appendix}

\usepackage{graphicx}
\usepackage{mathtools,array}
\usepackage[all]{xy}
\usepackage{tikz}
\usetikzlibrary{cd}
\usetikzlibrary{fit}
\usetikzlibrary{knots}
\usetikzlibrary{positioning}
\usetikzlibrary{arrows.meta}

\usepackage{fancyhdr}
\usepackage{upgreek,color}
\usepackage{hyperref}

\usepackage{blkarray}
\usepackage{multirow}

\usepackage{dutchcal} 

\usepackage[style=numeric,sorting=nyt,maxbibnames=7,maxcitenames=5,block=space,backend=bibtex]{biblatex}
\renewbibmacro{in:}{}

\hypersetup{linkbordercolor={0 0.60 0.75}, linkcolor=MidnightBlue}
\hypersetup{citebordercolor={0 0.60 0.75}, citecolor=MidnightBlue}

\numberwithin{equation}{section}
\theoremstyle{plain}

\newtheorem{thm}{Theorem}[section]

\newtheorem{lem}[thm]{Lemma}

\newtheorem{pps}[thm]{Proposition}

\theoremstyle{definition}
\newtheorem{rmkx}[thm]{Remark}
\newenvironment{rmk}
  {\pushQED{\qed}\renewcommand{\qedsymbol}{$\triangle$}\rmkx}
  {\popQED\endrmkx}

\newtheorem{prob}[thm]{Problem}

\declaretheoremstyle[
  spaceabove=-6pt,
  spacebelow=6pt,
  headfont=\normalfont\bfseries,
  postheadspace=1em,
  qed=\qedsymbol,
  headpunct={}
]{mystyle} 
\declaretheorem[name={proofof},style=mystyle,unnumbered,
]{proofof}  
\makeatletter
\renewenvironment{proofof}[1] {\par\pushQED{\qed}\normalfont\topsep6\p@\@plus6\p@\relax\trivlist  \item[\hskip\labelsep
        \bfseries
    Proof of #1.]\ignorespaces}{\popQED\endtrivlist\@endpefalse}
\makeatother
\makeatletter
\newcommand{\thickhline}{
    \noalign {\ifnum 0=`}\fi \hrule height 1pt
    \futurelet \reserved@a \@xhline
}
\newcolumntype{"}{@{\hskip\tabcolsep\vrule width 1pt\hskip\tabcolsep}}
\makeatother

\makeatletter
\renewenvironment{proof}[1][\proofname] {\par\pushQED{\qed}\normalfont\topsep6\p@\@plus6\p@\relax\trivlist\item[\hskip\labelsep\bfseries#1\@addpunct{.}]\ignorespaces}{\popQED\endtrivlist\@endpefalse}
\makeatother

\makeatletter
\renewcommand{\@secnumfont}{\bfseries}
\makeatother
\usepackage{etoolbox}
\makeatletter
\patchcmd{\section}{\scshape}{\bf}{}{}
\patchcmd{\subsection}{\scshape}{\bf}{}{}
\patchcmd{\subsubsection}{\scshape}{\bf}{}{}
\makeatother

\newcommand{\N}{\mathbb{N}} 
\newcommand{\Z}{\mathbb{Z}}

\newcommand{\Q}{\mathbb{Q}} 
 
\newcommand{\K}{\mathbb{K}} 
\newcommand{\F}{\mathbb{F}}

\newcommand{\laurent}[1]{#1[t,t^{-1}]} 
\newcommand{\OS}{\mathcal{O}_S} 

\newcommand{\anel}[1]{#1_{\mathrm{ring}}} 
\newcommand{\addi}[1]{#1_{\mathrm{add}}} 

\newcommand{\ri}{\mathcal{O}} 

\DeclareMathOperator{\GL}{GL} 
\DeclareMathOperator{\SL}{SL} 
\newcommand{\eij}{e_{i,j}} 
\newcommand{\di}{d_i} 
\newcommand{\dk}[1]{d_{#1}} 
\newcommand{\ekl}[1]{e_{#1}} 
\newcommand{\sbgpeij}{\mathcal{E}_{i,j}} 
\newcommand{\mbU}{\mathbf{U}} 
\newcommand{\mbD}{\mathbf{D}} 

\newcommand{\PGL}{\mathbb{P}\mathrm{GL}}

\newcommand{\Mult}{\mathbb{G}_{m}} 
\newcommand{\Addi}{\mathbb{G}_{a}} 
\newcommand{\Aff}{\mathbb{A}\mathrm{ff}} 
\newcommand{\mbB}{\mathbf{B}} 
\newcommand{\mbG}{\mathbf{G}} 
\newcommand{\PB}{\mathbb{P}\mathbf{B}} 

\newcommand{\PBplus}{\mathbb{P}B^{+}} 

\newcommand{\phee}{\varphi} 
\newcommand{\veps}{\varepsilon} 
\newcommand{\barra}[1]{\overline{#1}}

\newcommand{\til}[1]{\widetilde{#1}}
\newcommand{\gera}[1]{\langle {#1} \rangle}
\newcommand{\set}[1]{\{ #1 \}}

\newcommand{\mb}{\mathbf}
\newcommand{\mc}{\mathcal}

\newcommand{\into}{\hookrightarrow}
\newcommand{\onto}{\twoheadrightarrow}
\newcommand{\nsgp}{\trianglelefteq}

\newcommand{\Ri}{R_{\infty}} 
\newcommand{\RSpec}{\mathrm{RSpec}}

\newcommand{\FPn}[1]{\mathrm{FP}_{#1}}
\newcommand{\bref}[1]{{\bf(\ref{#1})}} 
\newcommand{\dref}[2]{\ref{#1}{\bf(\ref{#2})}} 

\DeclareMathOperator{\End}{End}
\DeclareMathOperator{\Aut}{Aut}
\DeclareMathOperator{\Inn}{Inn}
\DeclareMathOperator{\im}{Im}

\DeclareMathOperator{\carac}{char}

\DeclareMathOperator{\id}{id}

\title[Reidemeister numbers for soluble arithmetic groups in type A]{Reidemeister numbers for arithmetic Borel subgroups in type A}
\author{Paula Macedo Lins de Araujo and Yuri Santos Rego}
\address{Katholieke Universiteit Leuven, \newline 
Wiskunde, Campus Kulak, \newline 
Etienne Sabbelaan 53, bus 7657, \newline 
8500 Kortrijk, Belgi\"e}
\curraddr{University of Lincoln, 
Charlotte Scott Centre for Algebra, \newline 
Isaac Newton Building, Brayford Pool, LN6 7TS, Lincoln, United Kingdom}
\email{pmacedolinsdearaujo@lincoln.ac.uk}
\address{Otto-von-Guericke-Universit\"at Magdeburg, \newline 
Fakult\"at f\"ur Mathematik -- Institut f\"ur Algebra und Geometrie, \newline 
Postfach 4120, 39016 Magdeburg, Deutschland}
\email{ysantosrego@lincoln.ac.uk}
\curraddr{University of Lincoln, 
Charlotte Scott Centre for Algebra, \newline 
Isaac Newton Building, Brayford Pool, LN6 7TS, Lincoln, United Kingdom}

\begin{document}
\begin{abstract}
The Reidemeister number $R(\varphi)$ of a group automorphism $\varphi \in \mathrm{Aut}(G)$ encodes the number of orbits of the $\varphi$-twisted conjugation action of $G$ on itself, and the Reidemeister spectrum of $G$ is defined as the set of Reidemeister numbers of all of its automorphisms. We obtain a sufficient criterion for some groups of triangular matrices over integral domains to have property~$R_\infty$, which means that their Reidemeister spectrum equals $\{\infty\}$. Using this criterion, we show that Reidemeister numbers for certain soluble $S$-arithmetic groups behave differently from their linear algebraic counterparts --- contrasting with results of Steinberg, Bhunia, and Bose.  
\end{abstract}
\maketitle \vspace{-1.0cm}
\thispagestyle{empty}

\section{Introduction} 
Given a group~$G$ and an automorphism $\phee\in\Aut(G)$, the $\phee$\emph{-(twisted) conjugacy class} of $g\in G$ is the set of elements that are $\phee$-conjugate to $g$, 
\[ [g]_\phee = \set{h g \phee(h)^{-1} \mid h \in G}. \]
The \emph{Reidemeister number} of $\phee$, denoted by $R(\phee)$, is the total number of $\phee$-conjugacy classes (also called {Reidemeister classes} of $\phee$). In particular, the \emph{class number} of a group is just $R(\mathrm{id})$. 
The \emph{Reidemeister spectrum} $\RSpec(G)$ of the group $G$ is defined as the set 
\[\RSpec(G)= \{R(\phee) \mid \phee \in \Aut(G) \} \subseteq \Z_{\geq 1} \cup \{\infty\}\] 
of all possible Reidemeister numbers for $G$. Following Taback and Wong~\cite{TabackWong0}, we say that $G$ has \emph{property}~$\Ri$ when $\RSpec(G) = \{\infty\}$.

A main reason for interest in the question whether a group does or does not have $\Ri$ is the connection with fixed point statements in different areas; we refer the reader, e.g., to \cite[Theorem~6.1]{DacibergWongCrelle}, \cite[Theorem~10.1]{SteinbergEndo}, and \cite[Theorem~5.4]{Brasuca0} for examples in algebraic topology, Lie theory, and cohomology, respectively. 

Our first main theorem concerns the behavior of property~$\Ri$ within the family of $S$-arithmetic groups, so we now recall their definition. Let $\mbB$ be a linear algebraic group defined over a global field $\K$. A group $\Gamma$ is called an $S$-arithmetic subgroup of $\mbB(\K)$ if $\Gamma$ is commensurable with the subgroup $\rho^{-1}(\GL_n(\OS)) \leq \mbB(\K)$ for a ring of $S$-integers $\OS \subset \K$ and some faithful $\K$-representation $\rho : \mbB \into \GL_n$. 
The reader is referred to the classics \cite{Margulis,PlatonovRapinchuk} for major results and literature on such groups. We briefly mention that the family of $S$-arithmetic groups includes important familiar objects, such as finitely generated free groups, crystallographic groups, and many lattices in products of (real, complex, and $p$-adic) semisimple Lie groups.

Given $n\geq 2$, let $\PB_n$ denote the Borel subgroup of (upper) triangular matrices in the split semisimple group $\PGL_n$ of type $\mathtt{A}_{n-1}$. 
Our first theorem below answers negatively the arithmetic analoga of two questions arising from the study of Reidemeister spectra of connected soluble linear algebraic groups, initiated by Steinberg~\cite{SteinbergEndo} and continued recently by Bhunia and Bose \cite{BhuniaBose1,BhuniaBose2}; see below for the questions and a more thorough discussion.
%
%

\begin{thm}\label{thm:Lieapplication}
For each characteristic $p$ (zero or prime), there is a global field $\K_p$ of characteristic $\carac(\K_p)=p$ and $S$-arithmetic subgroups $\Gamma_{n,p} \leq \PB_n(\K_p)$ that satisfy the following: 
\begin{enumerate}
		\item $\RSpec(\Gamma_{2,p})$ contains some $m \notin \{1,\infty\}$. In particular, $\Gamma_{2,p}$ 
		does not have property~$R_\infty$, but 
		\item $\RSpec(\Gamma_{n,p}) = \{\infty\}$ for all $n \geq 4$.
\end{enumerate}
If $p \geq 5$, the groups $\Gamma_{n,p}$ can be chosen to be finitely generated. (If $\mathrm{char}(\K_{p}) = p = 0$, the groups $\Gamma_{n,p}$ are always finitely presented.) 

In case $p \in \{2,3\}$, one may take global fields $\K_{2,p}  \neq \K_{p}$ of characteristic $p$ and finitely generated $S$-arithmetic subgroups $\til{\Gamma}_{2,p} \leq \PB_2(\K_{2,p})$ and $\til{\Gamma}_{n,p} \leq \PB_n(\K_p)$, $n \geq 3$, satisfying~(i) and~(ii) above.
\end{thm}

Two points in the statement stand out. Firstly, the slight distinction in characteristics $p=2$ or $3$ to obtain finitely generated groups with the prescribed properties. Here, the original finitely generated candidates $\Gamma_{2,2} \leq \PB_2(\K_2)$ and $\Gamma_{2,3} \leq \PB_2(\K_3)$ actually have property~$R_\infty$, which we overcome by slightly changing the field; see the discussion at the end of Section~\ref{sec:backtoarithmetic}. 
Secondly, the omission of the case $n=3$ in Theorem~\ref{thm:Lieapplication} is an artifact of its proof, which relies on our other main theorem. Namely, we establish a criterion for detecting property $R_\infty$ in linear groups of upper triangular matrices of dimension at least four by reducing complexity of which automorphisms to look at. 
We now describe the matrix groups to which this tool applies.

%

%

For $n \geq 2$, let $\mbB_n \leq \GL_n$ denote the $\Z$-subscheme of invertible upper triangular matrices. For a (commutative) ring $R$ (with unity $1 \neq 0$), we work concretely with the groups of $R$-points 
\[ 
\mbB_n(R) = \left( \begin{smallmatrix} * & * & \cdots & * \\  & * & & \\ & & \ddots & \vdots \\ & & & * \\  \end{smallmatrix} \right) \leq \GL_n(R).
\]
We also consider their projective variants $\PB_n(R)$ obtained by modding out scalar matrices. Our main technical result is as follows.

\begin{thm} \label{thm:Additive}
Suppose $R$ is an integral domain with finitely generated group of units. Given a ring automorphism $\alpha \in \anel{\Aut}(R)$, write 
\begin{itemize}
	\item $\addi{\alpha}$ for the same map seen as an automorphism of $(R,+)$, and 
	\item $\tau_\alpha \in \Aut( R\times R)$ for the induced `flip' automorphism $\tau_\alpha((r,s)) = (\addi{\alpha}(s),\addi{\alpha}(r))$.
\end{itemize} If $R(\addi{\alpha}) = \infty = R(\tau_\alpha)$ for every $\alpha \in \anel{\Aut}(R)$, then $G(R)$ has property~$\Ri$ for all $G=\mbB_n(R)$ and all $G=\PB_n(R)$ with $n \geq 4$.
\end{thm}

The strength of this criterion lies in the achieved complexity reduction and on its reach. For a large class of rings, once 
a comparatively small collection of maps arising from 
ring automorphisms 
are shown to have infinitely many Reidemeister classes, every group $\mbB_n(R)$ and $\PB_n(R)$ with $n\geq 4$ automatically has property $R_\infty$. 
We provide in Section~\ref{applications} applications of Theorem~\ref{thm:Additive} in both non-arithmetic and arithmetic settings. That is, we use our criterion on upper triangular groups over the domains $\Z[t]$ and $\Z[t,t^{-1}]$ --- 
which are not rings of $S$-integers of any global field --- 
but also over the $S$-arithmetic rings 
$\F_p[t]$, $\F_p[t,t^{-1}]$, and rings of integers $\mathcal{O}_{\K}$ of algebraic number fields.

Let us elaborate a bit more on the motivation behind Theorem~\ref{thm:Lieapplication}. 
For a connected linear algebraic group $\mbB$ defined over an algebraically closed field, there is a strong relation between its structure and its Reidemeister spectrum, as investigated by Steinberg~\cite{SteinbergEndo} and more recently by Bhunia and Bose~\cite{BhuniaBose1,BhuniaBose2}. Indeed, if there is some \emph{algebraic} automorphism $\phee$ of $\mbB$ with $R(\phee) < \infty$, then $\mbB$ is necessarily soluble~\cite[Theorem~17]{BhuniaBose1}. As a partial converse, if $\mbB$ is soluble then $R(\phee)$ is either $1$ or $\infty$~\cite[Theorem~2.4]{BhuniaBose2}. 
And, in particular, when $\mbB \leq \mbG$ is a Borel subgroup of a semisimple group $\mbG$, one always has $R(\phee)=\infty$; cf. \cite[Theorem~2.10]{BhuniaBose2}.

It is therefore natural to investigate whether an analogous phenomenon appears in the discrete ($S$-arithemtic) setting. In particular, the following questions arise:
\begin{enumerate}
	\item \label{item:qi} How does the Reidemeister spectrum of an $S$-arithmetic group affect its (algebraic) structure? 
	\item \label{item:qii} If $\Gamma$ is an $S$-arithmetic subgroup of $\mbB(\K)$, with $\mbB$ a Borel subgroup of a semisimple group over a global field $\K$, does $\Gamma$ necessarily have property~$\Ri$? 
	\item \label{item:qiii} In case such a $\Gamma$ does not have~$\Ri$, is $\RSpec(\Gamma) = \{1, \infty\}$? 
\end{enumerate}

A partial answer to the first question follows from work of Jabara. Namely, if a finitely generated linear group $\Gamma$ has an automorphism $\phee$ of finite order with $R(\phee) < \infty$, then $\Gamma$ is virtually soluble; see \cite[Theorem~C]{Jabara}. 

Theorem~\ref{thm:Lieapplication} is our contribution to the theory, giving negative answers to the remaining two questions --- though only narrowly, since~$\Ri$ is in fact attained at higher derived length, in analogy with known nilpotent cases~\cite{DacibergWongCrelle,Romankov,KarelDacibergNil,TimurUniTri}. 

It should be stressed that Theorem~\ref{thm:Lieapplication} and the previously mentioned results are `uniform' in the sense that they hold in every characteristic. This is in strong contrast with other results for algebraic or $S$-arithmetic groups; cf.~\cite{LitterickThomasCRgood,Tiemeyer,Bux04}, for example. 

Our work extends, or improves on, some established results. 
In characteristic zero, a weaker form of Theorem~\ref{thm:Lieapplication} can be deduced from Nasybullov's work~\cite{TimurUniTri} by taking a family of virtually nilpotent arithmetic groups; cf. Remark~\ref{teoremadoTimur} for more details. However, requiring non-virtually-nilpotent examples and achieving~$\Ri$ in dimension $n=4$ onwards, as in our theorem, improves on Nasybullov's findings. In positive characteristic, examples of metabelian $S$-arithmetic groups as in part~(i) of Theorem~\ref{thm:Lieapplication} were known by work of Gon\c{c}alves--Wong~\cite{DacibergWongWreath}. The cases of arbitrarily high derived length in positive characteristic and further examples shown in Section~\ref{applications} are, to our knowledge, new. 

Putting our results further into perspective, Theorem~\ref{thm:Additive} is a contribution towards the wide problem of classifying which amenable groups have property~$\Ri$. There has been substantial progress in the (virtually) nilpotent \cite{DacibergWongCrelle,Romankov,KarelDacibergNil,TimurUniTri}, polycyclic \cite{KarelPenni,KarelSam,KarelSamIrisSolv}, and metabelian cases \cite{TabackWong0,DacibergWongWreath,SteinTabackWong,FelshtynDacibergMetabelian}. Moreover, non-polycyclic nilpotent-by-abelian groups of type~$\FPn{\infty}$ always have~$\Ri$; see \cite[Theorem~4.3]{DesiDaciberg}. 

While Theorem~\ref{thm:Additive} has intersections with some of the above mentioned results, it contributes towards a systematic investigation of twisted conjugacy of many soluble linear groups, elucidating how the structure of the underlying base ring leads to the given group having~$\Ri$. In a companion paper~\cite{Bn1}, whose first draft contained most of the present article, we further highlight how the structure of the base ring --- in that case, as a module over its units --- can detect~$\Ri$ for $\mbB_n(R)$ and its variants, starting with derived length $2$. In particular, it is possible to construct (in any characteristic $p \geq 0$) a family $\{\Gamma_{n,p}\}_{n \geq 2}$ of finitely {presented}, soluble, non-nilpotent, $S$-arithmetic groups $\Gamma_{n,p}$ with $\RSpec(\Gamma_{n,p}) = \{\infty\}$ whose derived lengths grow on $n$. 
Addressing aspects of twisted conjugacy arising from Nielsen fixed point theory, we also show in the paper~\cite{Bn1} how to construct --- using some of those arithmetic groups --- solvmanifolds of arbitrarily high dimensions all of whose self-homeomorphisms can be homotoped to become free of fixed points. 

The present paper also leaves open some questions. 
For instance, an expected strengthening of Theorem~\ref{thm:Lieapplication}, still in type $\mathtt{A}$, would be to drop the workarounds for characteristic two or three and also clear the case of dimension $n=3$. 
More precisely, 

\begin{prob}
Let an arbitrary characteristic $p \geq 0$ be given and let $\mbG_n$ denoted a Borel subgroup of $\GL_n$, $\PGL_n$ or $\SL_n$. Is there a single global field $\K$ of characteristic $\carac(\K)=p$ together with $S$-arithmetic subgroups $\Gamma_{n,p} \leq \mbG_n(\K)$ 
(of same derived length as $\mbG_n(\K)$) 
such that 
\[ \{\infty\} \neq \RSpec(\Gamma_{2,p}) \neq \{1,\infty\} \quad \text{ and } \quad \{\infty\} \neq \RSpec(\Gamma_{3,p}) \neq \{1,\infty\} \]
\[\text{ but } \quad \RSpec(\Gamma_{n,p}) = \{\infty\} \text{ for } n \geq 4? \]
\end{prob}

We remark that, although we explicitly compute Reidemeister numbers in some cases (see Section~\ref{sec:newexampleswithoutRinfty}) to establish part~(i) of Theorem~\ref{thm:Lieapplication}, we do not go so far as to determine the full Reidemeister spectrum of the groups without~$\Ri$ that we investigate here; see \cite{KarelPenni,FelshtynDacibergMetabelian,KarelKaiserSam,KarelSamIrisSolv} for examples in this direction. Similarly, we do not have descriptions for Reidemeister classes or asymptotic properties of Reidemeister numbers. These are interesting (and challenging) problems on their own. Related topics include problems around Higman's conjecture~\cite{HalasiPalfy,Higman,VLA}, dynamical zeta functions attached to Reidemeister numbers~\cite{DekTerVandeBus,FelshtynHill,FelshtynZietek}, and zeta functions of groups counting Reidemeister classes~\cite{PaulaII,ask,RoVo19,duSau05}. Lastly, we do not cover the other classical Lie types $\mathtt{B}_n$ to $\mathtt{G}_2$. Our results and the indicated literature thus motivate us to pose the following.

\begin{prob}
Let $\mc{G}$ be a split reductive linear algebraic group and fix a Borel subgroup $\mc{B} \leq \mc{G}$, all defined over some global field $\K$. Given an $S$-arithmetic subgroup $\Gamma$ of $\mc{B}(\K)$, describe the Reidemeister spectrum $\RSpec(\Gamma)$. And, in case $\RSpec(\Gamma) \neq \{\infty\}$, how do the Reidemeister numbers of automorphisms of $\Gamma$ grow?
\end{prob}

This paper is structured as follows. 
We start with Section~\ref{sec:LemmataRinfty} collecting well-known facts about Reidemeister numbers to be used throughout. In the same section we introduce our groups of interest in more detail. 
An important ingredient in the proof of Theorem~\ref{thm:Additive} is a description of automorphisms of the group of unitriangular matrices, completed by 
V.~Levchuk in~\cite{LevchukOriginal} 
--- we present his theorem in our terminology in Section~\ref{sec:Levchuk}, and sketch how to deduce it from~\cite{LevchukOriginal}. 

We restate and prove Theorem~\ref{thm:Additive} in Section~\ref{sec:provadothmB}. In Section~\ref{applications}, we illustrate how to apply it via non-trivial examples; see Proposition~\ref{pps:FptFptt-1OK}. 
Following up, we complement the previous findings with examples of groups \emph{without} property~$\Ri$, computing explicit Reidemeister numbers; cf. Proposition~\ref{nori}. Theorem~\ref{thm:Lieapplication} summarizes part of these findings, and the final Section~\ref{sec:backtoarithmetic} contains a proof of it for completeness.

To keep the article as self-contained as possible 
we do not assume familiarity with 
algebraic or $S$-arithmetic groups and $S$-arithmetic rings.  
We work concretely with the given matrix groups and base rings, so the reader familiar with standard results might want to skip directly to the the proof of the main theorem in Section~\ref{sec:provadothmB}, and the applications in Section~\ref{applications}.

\section{Auxiliary results, and structure of the groups}
\label{sec:LemmataRinfty}
We first recall known results on Reidemeister numbers, particularly for group extensions. Then, we collect some facts and notation concerning the groups considered in this paper.

\begin{lem}[{See \cite[Cor.~2.5]{FelshtynTroitskyCrelle}}] \label{lem:ignoreinner}
Let $G$ be a group and $\phee \in \Aut(G)$. Let also $\iota_g \in \mathrm{Inn}(G) \leq \Aut(G)$ denote the inner automorphism $\iota_g(h) = ghg^{-1}$. Then $R(\iota_g \circ \phee) = R(\phee)$.
\end{lem}

\begin{lem}[{See \cite[Prop.~1.2]{Daciberg}, \cite[Thm.~1]{WongCrelle}, \cite[Lem.~1.1(2)]{DacibergWongCrelle}}] \label{lem:desempre}
Suppose there is a short exact sequence of groups $N \into G \onto Q$ where $N$ is invariant under $\phee\in\Aut(G)$. Denote by $\phee' \in \Aut(N)$ and $\overline{\phee} \in \Aut(Q)$ the automorphisms induced by $\phee$, i.e., $\phee' = \phee|_N$ and, after fixing an isomorphism $Q \cong G/N$ coming from the exact sequence, $\barra{\phee}(gN) = \phee(g)N$. Then the following hold.

\begin{enumerate}
\item \label{PHEEvsPHEEbarra} $R(\phee) \geq R(\overline{\phee})$.
\item \label{PHEEbarra_and_innerPHEE'} 
If $R(\overline{\phee}) < \infty$ and if $R(\iota_g \circ \phee') < \infty$ for all inner automorphisms $\iota_g \in \Aut(G)$, 
then $R(\phee) < \infty$. 
\item \label{PHEEbarrafinite_and_PHEE'infinite} Suppose $\overline{\phee}$ has {finitely many} fixed points. If $R(\overline{\phee}) < \infty$ and $R(\phee') = \infty$, then $R(\phee) = \infty$.
\end{enumerate}
\end{lem}

\begin{lem}[{See \cite[Section~2 and Lemma~4.1]{Romankov}, \cite{KarelDacibergAbelian}}] \label{lem:ReidemeisterAbelian}
Let $A$ be a {finitely generated} abelian group and $\phee \in \Aut(A)$. 
Then $R(\phee) = \infty$ if and only if $\phee$ has {infinitely many} fixed points.
\end{lem}

We now turn to 
our groups of interest. 
For us, $R$ denotes a \emph{commutative} ring \emph{with unity} $1\neq 0$ unless stated otherwise. 

The group $\mbB_n(R)$, also called the (standard) \emph{Borel subgroup} of $\GL_n(R)$, is the soluble subgroup of $\GL_n(R)$ of upper triangular matrices. 

Let $\mbU_n(R) \leq \mbB_n(R)$ be the subgroup of unitriangular matrices, and $\mbD_n(R) \leq \mbB_n(R)$ 
the subgroup of diagonal matrices. As an abstract group, $\mbB_n(R)$ decomposes as the semi-direct product $\mbB_n(R) = \mbU_n(R) \rtimes \mbD_n(R)$. 
Since $R$ is commutative, the group $\mbU_n(R)$ is nilpotent (of nilpotency class $n-1$) and {$\mbD_n(R)$} 
is abelian, hence $\mbB_n(R)$ is soluble of derived length at most $n$.

Modding out scalar matrices leads to an important quotient. Specifically, 
consider the central subgroup 
\[Z_n(R) = \set{u\cdot \mb{1}_n \in \GL_n(R) \mid u \in R^\times},\] 
where $\mb{1}_n$ denotes the $n \times n$ identity matrix.  
The \emph{projective upper triangular group} $\PB_n(R)$ 
is given by 
\[ \PB_n(R) = \frac{\mbB_n(R)}{Z_n(R)}.\] 

\begin{rmk} \label{obs:centerBn} 
{If $R$ is an integral domain, straightforward computations show that $Z_n(R)$ coincides with the center of $\mbB_n(R)$ whenever it is non-trivial.} 
Thus $\PB_n(R)$ 
is a quotient of $\mbB_n(R)$ 
modulo a {characteristic} subgroup whenever $R$ is an integral domain. 
\end{rmk}

\begin{rmk} \label{obs:referee} Note that $Z_n(R) \subseteq \mbD_n(R)$ and that the restriction of the canonical quotient map 
\[
\pi : \mathbf{B}_n(R) \longrightarrow \PB_n(R) = \mathbf{B}_n(R)/Z_n(R)
\]
to $\mathbf{U}_n(R)$ is injective, which yields an obvious isomorphism
\begin{equation}
\label{eq:referee} \PB_n(R) \;\cong\; \mathbf{U}_n(R) \rtimes 
\bigl(\mathbf{D}_n(R)/Z_n(R)\bigr) 
\end{equation}
inherited from the semidirect product 
$\mathbf{B}_n(R)=\mathbf{U}_n(R)\rtimes \mathbf{D}_n(R)$.

By abuse of notation, we identify $\mbU_n(R)$ with its canonical copy in $\PB_n(R)$ under the map $\pi$, so that any element $g \in \PB_n(R)$ 
can be uniquely written as $g = u[d]$ where $u \in \mbU_n(R)$ and $[d]$ is the class of $d \in \mbD_n(R)$ in $\mbD_n(R)/Z_n(R)$. 
\end{rmk}

We remark that $\PB_2(R)$ is isomorphic to the group $\Aff(R) \cong \left( \begin{smallmatrix} * & * \\ 0 & 1 \end{smallmatrix} \right) \leq \GL_2(R)$ of affine transformations of the base ring $R$; cf. \cite[Lemma~3.6]{Bn1}.


We introduce some more notation that will be used throughout. We let $\Addi(R) = (R,+)$ and $\Mult(R) = (R^\times, \cdot)$ denote the underlying additive group and the group of units of the ring $R$, respectively.

Following common notation for $\GL_n(R)$, we write $\eij(r)$ to denote the usual elementary matrix with off-diagonal entry $r\in R$ in position $(i,j)$, all its diagonal entries equal to~$1$, and having zeroes elsewhere. As we work with upper triangular groups throughout, only such matrices with $i<j$ will be used. 
For instance, in $\GL_2(\Z)$ one has $\ekl{1,2}(2) = \left( \begin{smallmatrix} 1 & 2 \\ 0 & 1 \end{smallmatrix} \right)$. For a fixed $(i,j)$ we also denote $\mc{E}_{i,j}(R) := \langle \{ \ekl{i,j}(r) \in \mbU_n(R) \mid r \in R \}\rangle \cong \Addi(R)$.

Given $i \in \set{1,\ldots,n}$ and $u \in R^\times$ we let $\di(u) \in \GL_n(R)$ denote the diagonal matrix whose $i$-th entry is $u$ and all other (diagonal) entries are equal to~$1$. For example, $\dk{2}(-1) = \left( \begin{smallmatrix} 1 & 0 \\ 0 & -1 \end{smallmatrix} \right)$ in $\GL_2(\Z)$. 
The elements 
$\eij(r)$ are called \emph{elementary} matrices and the $\di(u)$ are 
\emph{elementary diagonal} matrices.

We recall that elements of $\mbU_n(R)$ can be uniquely written as a product of elementary matrices ordered according to the superdiagonals of $\mbU_n(R)$. More precisely, given $x \in \mbU_n(R)$, there exist (uniquely determined) $r_{i,j} \in R$ such that 
\begin{align} \label{rel:elementsofUn} 
\begin{split}
x = & \phantom{.} e_{1,2}(r_{1,2}) e_{2,3}(r_{2,3}) \cdot \ldots \cdot e_{n-1,n}(r_{n-1,n}) \cdot \\
 & \cdot  e_{1,3}(r_{1,3}) \cdot \ldots \cdot e_{n-2,n}(r_{n-2,n}) \cdot \ldots \ldots \cdot e_{1,n}(r_{1,n}). 
\end{split}
\end{align}

All of the group theoretical properties above for $\mbB_n(R)$, $\mbD_n(R)$ and $\mbU_n(R)$ are deduced from classic observations using the following sets of relations; see, e.g., \cite{HahnO'Meara, Silvester}. 
(Group commutators in this paper are written as $[g,h] = ghg^{-1}h^{-1}$.) 
\begin{align}
\label{rel:commutators}
\begin{split}
[\eij(r),\ekl{k,l}(s)] & =
\begin{cases}
\ekl{i,l}(rs) & \mbox{if } j=k,\\
1 & \mbox{if } i \neq l \text{ and } k \neq j,
\end{cases} \\
\di(u) \ekl{k,l}(r) \di(u)^{-1} & = 
\begin{cases}
\ekl{k,l}(ur) & \mbox{if } i=k,\\
\ekl{k,l}(u^{-1}r) & \mbox{if } i=l, \\
e_{k,l}(r) & \mbox{otherwise}.
\end{cases}
\end{split}
\end{align}
The equations~\bref{rel:commutators} are also referred to as \emph{elementary} or \emph{commutator relations}.
%
They also yield the following relations in the diagonal of $\PB_n(R)$. 
\begin{align} \label{rel:projectiveconjugation}
[d] \ekl{k,l}(r) [d]^{-1} = \ekl{k,l}(u_k u_l^{-1} r) & \phantom{a} & \text{ if } \quad d = d_1(u_1) \cdots d_n(u_n).
\end{align}

In the beginning of Section~\ref{sec:LemmataRinfty} we recalled useful results on the relationship between the (non-)finiteness of Reidemeister numbers and characteristic subgroups. 
In the category of linear algebraic groups (over a fixed field $\K$), one has that $\mbU_n$ is the commutator subgroup scheme of $\mbB_n$, hence (algebraically) characteristic in $\mbB_n$. In the general abstract case one has the following.
\begin{rmk} \label{rmk:Uncharacteristic}
The subgroup $\mbU_n(R)$ is not necessarily characteristic in $\mbB_n(R)$, even over integral domains; cf.~\cite[Section~3.2]{Bn1}. Regardless, if $R$ is an integral domain, $\mbU_n(R)$ is the \emph{Hirsch--Plotkin radical} (hence a characteristic subgroup) of $\PB_n(R)$; see, e.g., \cite[Proposition~3.9]{Bn1} for a proof.
\end{rmk} 

\section{The theorem of Levchuk} \label{sec:Levchuk}

The main ingredient in the proof of Theorem~\ref{thm:Additive} is a remarkable result due to Vladimir 
Levchuk 
describing the automorphisms of the group of unitriangular matrices over arbitrary (associative, unital) rings; see~\cite{LevchukOriginal}. 
We begin by recasting his theorem in the terminology and form needed for the present work --- in particular, we only work with integral domains and restate the results directly for the {upper} unitriangular matrix group $\mbU_n(R)$ with the usual matrix multiplication, instead of the group of {lower} unitriangular matrices $\mbU^{-}_n(R)$ as in the original paper~\cite{LevchukOriginal}. 

\begin{thm}[{Levchuk~\cite{LevchukOriginal}}] \label{thm:Levchuk} 
Let $R$ be an integral domain such that $R \neq \F_2$ and let $n \geq 4$. Then any automorphism $\psi \in \Aut(\mbU_n(R))$ can be written as a product 
\[\psi = \iota_u \circ \iota_{d} \circ \mc{z} \circ \sigma \circ \tau^\veps \circ \alpha_\ast, \]
 where 
 \begin{itemize}
     \item $\iota_u \in \Inn(\mbU_n(R))$ is conjugation by an element of $\mbU_n(R)$, and $\iota_d$ is conjugation by a diagonal matrix $d \in \mbD_n(R)$; see Section~\ref{Levchuk:inner},
     \item $\mc{z} \in \mc{Z}$ is a central automorphism; see Section~\ref{Levchuk:central}, 
     \item $\sigma \in \mc{U}^{(c)}$ is an extremal automorphism; 
		see Section~\ref{Levchuk:sigma}, 
     \item $\tau$ is the flip automorphism and $\veps \in \set{0,1}$; see Section~\ref{Levchuk:flip}, and 
     \item $\alpha_\ast$ is induced by a ring automorphism $\alpha \in \anel{\Aut}(R)$; see Section~\ref{Levchuk:ring}.
 \end{itemize}
\end{thm}

In what follows we elucidate the automorphisms appearing in the statement and, since the theorem does not appear in the original paper~\cite{LevchukOriginal} in the form stated above, we explain in Section~\ref{Levchuk:proof} how it follows from 
Levchuk's 
work. 

\subsection{Conjugation by elements of $\mbU_n(R)$ and $\mbD_n(R)$} \label{Levchuk:inner}
The first elements of $\Aut(\mbU_n(R))$ in 
Levchuk's 
list are akin to inner automorphisms, i.e., conjugation. In an arbitrary group $G$, given $g\in G$ we usually denote by $\iota_g$ or by $\kappa_g$ the inner automorphism $\iota_g(h) = ghg^{-1}$ (resp. $\kappa_g(h) = ghg^{-1}$). (We remind the reader that the group of inner automorphisms $\Inn(G)$ is isomorphic to $G/Z(G)$.) 
We then have the obvious subgroup $\Inn(\mbU_n(R)) \leq \Aut(\mbU_n(R))$ of inner automorphisms of $\mbU_n(R)$. However, since $\mbU_n(R)$ is also invariant under conjugation by {diagonal} matrices $d \in \mbD_n(R)$ due to the relations~\bref{rel:commutators}, 
it follows that the assignments $\iota_d(x) = dxd^{-1}$ for $d \in \mbD_n(R)$ and $x \in \mbU_n(R)$ may still be canonically viewed as automorphisms of $\mbU_n(R)$. (We are henceforth abusing notation using $\iota$ for such maps even though they are not inner automorphisms of $\mbU_n(R)$.) Levchuk denotes by $D$ the subgroup of $\Aut(\mbU_n(R))$ generated by all such $\iota_d$. In symbols,
\[D = \{ \iota_d \in \Aut(\mbU_n(R)) \mid d \in \mbD_n(R) \} \quad \text{ where } \quad \iota_d(x) = dxd^{-1}.\]
We stress, however, that every scalar matrix $d \in Z_n(R) = \{u \cdot \mb{1}_n \mid u \in R^\times\}$ yields $\iota_d = \id\vert_{\mbU_n(R)}$. 


\begin{rmk} \label{obs:innerautomorphisms}
Since scalar matrices act trivially on $\mbU_n(R)$ by conjugation, every class $[d] \in \mbD_n(R) / Z_n(R)$ of diagonal matrix modulo scalar matrices gives rise to a well-defined automorphism
\[\iota_{[d]}(x) = [d]x[d]^{-1}, \quad x \in \mbU_n(R),\]
of $\mbU_n(R)$; cf. relations~\bref{rel:projectiveconjugation}. Moreover, relations~\bref{rel:commutators} and~\bref{rel:projectiveconjugation} imply that $\iota_d = \iota_{[d]}$ for any $d \in \mbD_n(R)$. In particular, the subgroup $D \leq \Aut(\mbU_n(R))$ may be rewritten as being generated by those $\iota_{[d]}$ instead. 

Recalling that 
$\PB_n(R)$ is identified with $\mbU_n(R) \rtimes (\mbD_n(R) / Z_n(R))$ and given an element $g \in \PB_n(R)$ written uniquely as $g = u[d]$ with $u \in \mbU_n(R)$ and 
$[d] \in \mbD_n(R) / Z_n(R)$, it is clear that the subset of automorphisms $\Inn(\mbU_n(R)) \cdot D \subseteq \Aut(\mbU_n(R))$ coincides with the obvious image of $\Inn(\PB_n(R))$ in $\Aut(\mbU_n(R))$ obtained by restricting inner automorphisms of $\PB_n(R)$ to its normal subgroup $\mbU_n(R)$. But note that this natural map $\Inn(\PB_n(R)) \to \Aut(\mbU_n(R))$ is not injective in general since its kernel contains an isomorphic copy of $Z(\mbU_n(R))$. 
\end{rmk}

\subsection{Central automophisms} \label{Levchuk:central}
Next we consider automorphisms which only modify the center. 
Recall that the lower central series of a group $G$ is given by $\gamma_1(G) = G$ and $\gamma_{i+1}(G) = [\gamma_i(G),G]$. And if $G$ is non-trivial and nilpotent, its center $Z(G)$ always contains the last non-trivial term of the lower central series. 

Straightforward calculations (e.g., using relations~\bref{rel:commutators}) yield 
\[Z(\mbU_n(R)) = \gamma_{n-1}(\mbU_n(R)) = \mc{E}_{1,n}(R):=\langle \{ \ekl{1,n}(r) \in \mbU_n(R) \mid r \in R \}\rangle.\] 
Given an index $k \in \set{1,\ldots,n-1}$, any endomorphism $\lambda \in \End(\Addi(R))$ of the underlying additive group of $R$ gives rise to an element $\zeta_k(\lambda) \in \Aut(\mbU_n(R))$ as in~\cite[65]{LevchukOriginal}, defined by
\[ \zeta_k(\lambda)((a_{i,j})) = (a_{i,j}) \cdot \ekl{1,n}(\lambda(a_{i,i+1})) \, \text{ for every matrix } \, (a_{i,j}) \in \mbU_n(R). \]
The subgroup
\[ \mc{Z} = \gera{ \, \set{ \, \zeta_k(\lambda) \in \Aut(\mbU_n(R)) \, \mid \, k \in \set{1,\ldots,n-1}, \lambda \in \End(\Addi(R)) \, } \, } \]
is called the \emph{group of central automorphisms}.

\begin{rmk} \label{obs:centralautomorphisms}
Since the elementary matrices $\ekl{1,n}(r)$ are central in $\mbU_n(R)$, central automorphisms commute with inner automorphisms of $\mbU_n(R)$. That is, the subgroups $\mc{Z} \leq \Aut(\mbU_n(R))$ and $\Inn(\mbU_n(R)) \leq \Aut(\mbU_n(R))$ commute elementwise. (Though $\mc{Z}$ need not commute elementwise with $D$.)
\end{rmk}

\subsection{Extremal automorphisms}\label{Levchuk:sigma}
We now describe 
\emph{extremal automorphisms} of $\mbU_n(R)$. Such maps have been studied at least since the 1950s (see, e.g., Pavlov's work \cite{PavlovSylow}). For those familiar with linear algebraic groups, extremal automorphisms are so called because they are defined as maps acting on some simple roots located in the extrema of the underlying root system; cf. \cite{GibbsAut}. In our case, these are the first and last simple roots in type $\mathtt{A}_{n-1}$.

Let a function $\lambda: R \to R$ and an element $a \in R$ satisfy 
\begin{equation} \label{defsigmas} \lambda(r+s) = ars + \lambda(r) + \lambda(s) \, \text{ for all } \, r,s \in R. \end{equation} 
For example, when $R$ is a field with $\mathrm{char}(R)\neq 2$, a common choice for $\lambda$ and $a$ fulfilling Condition~\bref{defsigmas} is $\lambda(x) = -\frac{1}{2}x^2$ and $a=-1$.

Then, as seen in~\cite[66]{LevchukOriginal}, the maps $\sigma_{\lambda,a}$ and $\sigma_{\lambda,a}'$, defined on the generators of $\mbU_n(R)$ via
\[ \sigma_{\lambda,a} : \begin{cases} e_{1,2}(r) & \mapsto e_{1,2}(r) \cdot e_{2,n}(ar) \cdot e_{1,n}(\lambda(r) - ar^2) \\ e_{i,i+1}(r) & \mapsto e_{i,i+1}(r) \, \, \text{ if } i > 1, \end{cases} \]
\[ \sigma_{\lambda,a}' : \begin{cases} e_{n-1,n}(r) & \mapsto e_{n-1,n}(r) \cdot e_{1,n-1}(ar) \cdot e_{1,n}(\lambda(r)) \\ e_{i,i+1}(r) & \mapsto e_{i,i+1}(r) \, \, \text{ if } i < n-1, \end{cases} \]
induce automorphisms of $\mbU_n(R)$. 
For instance, if $n=3$ and $R = \Q$, the choices $\lambda(x) = -\frac{1}{2}x^2$ and $a=-1$ yield
\[\sigma_{\lambda,a}(e_{1,2}(r)) = e_{1,2}(r) e_{2,3}(-r) e_{1,3}\left(\frac{1}{2}r^2\right).\]
One quickly checks with the commutator relations~\bref{rel:commutators} that the above $\sigma_{\lambda,a}$ is an honest automorphism of $\mbU_3(\Q)$.

Abusing notation, we shall also denote all such automorphisms as above by $\sigma_{\lambda,a}$ and $\sigma_{\lambda,a}'$, and we call them \emph{extremal automorphisms} of $\mbU_n(R)$, borrowing terminology used by Pavlov \cite{PavlovSylow} and Gibbs \cite{GibbsAut}. 
Levchuk 
denotes by $\mc{U}^{(c)}$ the group generated by all extremal automorphisms, i.e., 
\[ \mc{U}^{(c)} = \gera{ \, \set{ \,\sigma_{\lambda,a}, \sigma_{\lambda,a}' \in \Aut(\mbU_n(R)) \, \mid \, \lambda \text{ and } a \text{ satisfy Condition~\bref{defsigmas}} \, } \, }. \]
For simplicity, we also call $\mc{U}^{(c)} \leq \Aut(\mbU_n(R))$ itself the group of extremal automorphisms, and its 
elements 
will typically be denoted by $\sigma$.

\subsection{Flip automorphism}\label{Levchuk:flip}
Next in our list is the `flip' automorphism, to be denoted by $\tau \in \Aut(\mbU_n(R))$. 
(This was termed `mirror' automorphism by Pavlov \cite{PavlovSylow}.) 
The reader familiar with Chevalley--Demazure groups might know $\tau$ as the restriction to $\mbU_n(R)$ of the automorphism of $\SL_n(R)$ induced by the unique graph automorphism of the Dynkin diagram of type $\mathtt{A}_{n-1}$; cf. \cite[206]{GibbsAut}, \cite[65]{LevchukOriginal}, and \cite[91]{Steinberg}. In our case, $\tau$ is most conveniently described by the rule
\[ \tau(\eij(r)) = e_{n-j+1,n-i+1}((-1)^{j-i-1}r) \, \text{ for } \eij(r) \in \mbU_n(R) \]
on the generators of $\mbU_n(R)$. 
Pictorially, $\tau$ is obtained by flipping the entries of the matrices in $\mbU_n(R)$ along the anti-diagonal and adjusting the signs of the matrix entries in accordance with the commutator relations~\bref{rel:commutators}. For instance, for $n=3$,
the flip is defined 
by
\begin{align*}
\left(\begin{smallmatrix}
    1 & r & 0 \\ 0 & 1 & 0 \\ 0 & 0 & 1
\end{smallmatrix}\right)&\xmapsto{\tau}\left(\begin{smallmatrix}
    1 & 0 & 0 \\ 0 & 1 & r \\ 0 & 0 & 1
\end{smallmatrix}\right), \quad 
\left(\begin{smallmatrix}
    1 & 0 & 0 \\ 0 & 1 & r \\ 0 & 0 & 1
\end{smallmatrix}\right)
\xmapsto{\tau}\left(\begin{smallmatrix}
    1 & r & 0 \\ 0 & 1 & 0 \\ 0 & 0 & 1
\end{smallmatrix}\right), \quad 
\left(\begin{smallmatrix}
    1 & 0 & r \\ 0 & 1 & 0 \\ 0 & 0 & 1
\end{smallmatrix}\right)
\xmapsto{\tau}\left(\begin{smallmatrix}
    1 & 0 & -r \\ 0 & 1 & 0 \\ 0 & 0 & 1
\end{smallmatrix}\right),
%
\end{align*}
whereas for $n=4$ one has 
\begin{align*}
\left(\begin{smallmatrix}
    1 & r & 0 & 0 \\ 0 & 1 & 0 & 0 \\ 0 & 0 & 1 & 0 \\0 & 0 & 0 & 1 
\end{smallmatrix}\right)&
\xmapsto{\tau}\left(\begin{smallmatrix}
    1 & 0 & 0 & 0 \\ 0 & 1 & 0 & 0 \\ 0 & 0 & 1 & r \\ 0 & 0 & 0 & 1
\end{smallmatrix}\right), \, \,
\left(\begin{smallmatrix}
    1 & 0 & 0 & 0 \\ 0 & 1 & r & 0 \\ 0 & 0 & 1 & 0 \\0 & 0 & 0 & 1 
\end{smallmatrix}\right)
\xmapsto{\tau}\left(\begin{smallmatrix}
    1 & 0 & 0 & 0 \\ 0 & 1 & r & 0 \\ 0 & 0 & 1 & 0 \\ 0 & 0 & 0 & 1
\end{smallmatrix}\right), \, \,
\left(\begin{smallmatrix}
    1 & 0 & 0 & 0 \\ 0 & 1 & 0 & 0 \\ 0 & 0 & 1 & r \\0 & 0 & 0 & 1 
\end{smallmatrix}\right)
\xmapsto{\tau}\left(\begin{smallmatrix}
    1 & r & 0 & 0 \\ 0 & 1 & 0 & 0 \\ 0 & 0 & 1 & 0 \\ 0 & 0 & 0 & 1
\end{smallmatrix}\right), 
\end{align*}
\begin{align*}
\left(\begin{smallmatrix}
    1 & 0 & r & 0 \\ 0 & 1 & 0 & 0 \\ 0 & 0 & 1 & 0 \\0 & 0 & 0 & 1 
\end{smallmatrix}\right)&
\xmapsto{\tau}\left(\begin{smallmatrix}
    1 & 0 & 0 & 0 \\ 0 & 1 & 0 & -r \\ 0 & 0 & 1 & 0 \\ 0 & 0 & 0 & 1
\end{smallmatrix}\right), \, \,
\left(\begin{smallmatrix}
    1 & 0 & 0 & 0 \\ 0 & 1 & 0 & r \\ 0 & 0 & 1 & 0 \\0 & 0 & 0 & 1 
\end{smallmatrix}\right)
\xmapsto{\tau}\left(\begin{smallmatrix}
    1 & 0 & -r & 0 \\ 0 & 1 & 0 & 0 \\ 0 & 0 & 1 & 0 \\ 0 & 0 & 0 & 1
\end{smallmatrix}\right), \, \,
\left(\begin{smallmatrix}
    1 & 0 & 0 & r \\ 0 & 1 & 0 & 0 \\ 0 & 0 & 1 & 0 \\0 & 0 & 0 & 1 
\end{smallmatrix}\right)
\xmapsto{\tau}\left(\begin{smallmatrix}
    1 & 0 & 0 & r \\ 0 & 1 & 0 & 0 \\ 0 & 0 & 1 & 0 \\ 0 & 0 & 0 & 1
\end{smallmatrix}\right). 
\end{align*}
It is immediate that $\tau^2 = \id$.

\subsection{Group automorphisms induced by ring automorphisms}\label{Levchuk:ring}
To finish our list of automorphisms of $\mbU_n(R)$ recall that, since $\mbU_n$ is an affine $\Z$-subscheme of $\GL_n$, any ring automorphism $\alpha \in \anel{\Aut}(R)$ induces by functoriality a group automorphism $\alpha_\ast : \mbU_n(R) \to \mbU_n(R)$. More explicitly, given $\alpha \in \anel{\Aut}(R)$ and a matrix $(a_{ij}) \in \mbU_n(R)$, one has $\alpha_\ast((a_{ij})) := (\alpha(a_{ij}))$. With this in mind we consider $\anel{\Aut}(R)$ as a subgroup of $\Aut(\mbU_n(R))$ in the obvious way and call it the \emph{subgroup of ring automorphisms} of $\mbU_n(R)$.

\begin{rmk} \label{obs:ringautomorphisms}
All subgroups $\sbgpeij(R) \leq \mbU_n(R)$ are $\alpha_\ast$-invariant for any ring automorphism $\alpha \in \anel{\Aut}(R)$. This is because ring automorphisms preserve~$1$ and so $\alpha_\ast(\eij(r)) = \eij(\alpha(r))$. In particular, $\alpha_\ast$ induces an automorphism $\barra{\alpha}_\ast$ on the abelianization $\mbU_n(R)^{\mathrm{ab}}$ such that the image of every $\mc{E}_{i,i+1}(R)$ in $\mbU_n(R)^{\mathrm{ab}}$ is also $\barra{\alpha}_\ast$-invariant. Moreover, since $\sbgpeij(R) \cong \Addi(R)$ we may write $\alpha_\ast(\eij(r)) = \eij(\addi{\alpha}(r))$, where $\addi{\alpha}$ is the same map $\alpha$ but viewed as an automorphism of the underlying additive group $\Addi(R) = (R,+)$. In particular, the restriction $\alpha_\ast|_{\sbgpeij(R)}$ is also interpreted as an automorphism of $\Addi(R)$.
\end{rmk}

\subsection{Proof of Theorem~\ref{thm:Levchuk}}\label{Levchuk:proof}

Here we spell out how Theorem~\ref{thm:Levchuk} follows from 
Levchuk's 
work. We want to show that $\Aut(\mbU_n(R))$ can be written as a product
\[\Aut(\mbU_n(R)) = \Inn(\mbU_n(R)) \cdot D \cdot \mc{Z} \cdot \mc{U}^{(c)} \cdot \gera{\tau} \cdot \anel{\Aut}(R)\]
whenever $n\geq 4$, which implies that any $\psi \in \Aut(\mbU_n(R))$ can be written as a product 
\[\psi = \iota_u \circ \iota_d \circ \mc{z} \circ \sigma \circ \tau^\veps \circ \alpha_\ast,\] 
with $\iota_u \in \Inn(\mbU_n(R))$, $\iota_d \in D$ being conjugation by a diagonal matrix, $\mc{z} \in \mc{Z}$ a central automorphism, $\sigma \in \mc{U}^{(c)}$ an extremal automorphism, $\varepsilon \in \{0,1\}$, $\tau$ the flip automorphism, and $\alpha_\ast$ is induced by a ring automorphism. We need to address the cases $n\geq 5$ and $n=4$ separately.

Recall that $D \leq \Aut(\mbU_n(R))$ is the subgroup of automorphisms which are given by conjugation by a diagonal matrix. By~\cite[Corollary~3]{LevchukOriginal}, if $n \geq 5$ then there exist two further subgroups $W$ and $V$ of $\Aut(\mbU_n(R))$ such that $\Aut(\mbU_n(R))$ decomposes as 
\[ \Aut(\mbU_n(R)) = (((\mc{Z} \cdot \Inn(\mbU_n(R)) \cdot \mc{U}^{(c)} \cdot V) \rtimes W) \rtimes D) \rtimes \anel{\Aut}(R). \]
(In \cite{LevchukOriginal}, the group $\Inn(\mbU_n(R))$ is denoted by a calligraphic J.)

The generators of the subgroup $V \leq \Aut(\mbU_n(R))$ above are induced by the assignments $\eta_a$, $\eta_b$ defined in~\cite[66]{LevchukOriginal}. The former are defined by choosing elements $a \in R$ belonging ([\emph{loc. cit.}]) to the left-annihilator of the set 
\[\set{rs-sr \mid r,s \in R} \cup \set{2} \cup \{ (x^2-x)(y^2-y) \mid x,y \in R\} \subseteq R.\]
(We remind the reader that Levchuk works with arbitrary associative unital rings.) 
Since our ring $R$ is commutative, the first set above is just $\{0\}$. Moreover, the last set must contain some non-zero element unless $R = \F_2$, which is excluded from our hypotheses. Thus $a$ would be a zero divisor, which is also not allowed as $R$ is an integral domain for us. Hence the only possible choice is $a=0$, in which case one readily checks that Levchuk's $\eta_0$ induces the identity map on $\mbU_n(R)$. The latter maps $\eta_b$ are defined via similar conditions by symmetry with right multiplication (\cite[66]{LevchukOriginal}), and the same reasoning with our hypotheses implies that the maps $\eta_b$ can be ignored as well, hence $V = \set{\id}$. 

The generators of the subgroup $W \leq \Aut(\mbU_n(R))$ are the so-called idempotent automorphisms $\tau_e$, described in~\cite[65]{LevchukOriginal}. The definition of such a generator $\tau_e \in W$ depends on the choice of an idempotent element $e$ lying in the center $Z(R)$ of the ring $R$ [\emph{loc. cit.}]. Here we stress that the flip automorphism $\tau$ coincides with the generator $\tau_0 \in W$ associated to the trivial idempotent $e = 0$. Since $R$ is an integral domain we have $Z(R) = R$ and it furthermore only admits the trivial idempotents $0$ and $1$. However, the description in~[\emph{loc. cit.}] also shows that $\tau_1$ induces the identity map. Thus in our case $W = \gera{\tau}$, which is a group of order $2$.

Therefore \cite[Corollary~3]{LevchukOriginal} applied to integral domains distinct from $\F_2$ actually yields
\begin{equation} \label{eq:Levchuk} \Aut(\mbU_n(R)) = (((\mc{Z} \cdot \Inn(\mbU_n(R)) \cdot \mc{U}^{(c)}) \rtimes \gera{\tau}) \rtimes D) \rtimes \anel{\Aut}(R) \end{equation}
in the case $n \geq 5$.

Now, recalling that a semi-direct product $G = N \rtimes Q$ can be equivalently written as $G = Q \ltimes N$ and moreover that $G = N \cdot Q = Q \cdot N$, we may rewrite equality~\bref{eq:Levchuk} in the desired way. Indeed, since $\mc{Z}$ and $\Inn(\mbU_n(R))$ commute by Remark~\ref{obs:centralautomorphisms}, equality~\bref{eq:Levchuk} gives 
\begin{align*} \Aut(\mbU_n(R)) & = ( D \ltimes ((\Inn(\mbU_n(R)) \cdot \mc{Z} \cdot \mc{U}^{(c)}) \rtimes \gera{\tau}) ) \rtimes \anel{\Aut}(R) \\
& = D \cdot \Inn(\mbU_n(R)) \cdot \mc{Z} \cdot \mc{U}^{(c)} \cdot \gera{\tau} \cdot \anel{\Aut}(R).
\end{align*}
Lastly, since $\mbU_n(R)$ is invariant under the conjugation action of $\mbD_n(R)$, it is clear that that a composition of conjugations $\iota_d \circ \iota_x \in D \cdot \Inn(\mbU_n(R)) \subseteq \Aut(\mbU_n(R))$ --- where $d \in \mbD_n(R)$ is a diagonal matrix and $x \in \mbU_n(R)$ is a unipotent matrix --- may be rewritten as $\iota_y \circ \iota_{d} \in \Inn(\mbU_n(R)) \cdot D$. Simply set $y = dxd^{-1}$, so that 
\begin{align*}
\iota_d \circ \iota_x (a) & = \iota_d (xax^{-1}) = dxax^{-1}d^{-1} \\
& = (d x d^{-1}) d a d^{-1} (d x d^{-1})^{-1} = \iota_{dxd^{-1}} \circ \iota_y (a)
\end{align*}
for all $a \in \mbU_n(R)$. Therefore we may swap the factors $D$ and $\Inn(\mbU_n(R))$ in the previous description of $\Aut(\mbU_n(R))$ and obtain 
\begin{align*} \Aut(\mbU_n(R)) & = \Inn(\mbU_n(R)) \cdot D \cdot \mc{Z} \cdot \mc{U}^{(c)} \cdot \gera{\tau} \cdot \anel{\Aut}(R),
\end{align*}
as desired.

We now address the case $n = 4$. The starting point is a description of $\Aut(\mbU_4(R))$ which is slightly similar to the one from the case $n \geq 5$. 
More precisely, by \cite[Theorem~2]{LevchukOriginal} there exists a subgroup $\widetilde{S} \leq \Aut(\mbU_4(R))$ such that $\Aut(\mbU_4(R))$ decomposes as 
\[ \Aut(\mbU_4(R)) = ((\mc{Z} \cdot \Inn(\mbU_4(R)) \cdot \mc{U}^{(c)}) \rtimes (\widetilde{S} \cdot D)) \rtimes \anel{\Aut}(R),\]
which we may rewrite as 
\[ \Aut(\mbU_4(R)) = ((\widetilde{S} \cdot D) \ltimes (\mc{Z} \cdot \Inn(\mbU_4(R)) \cdot \mc{U}^{(c)})) \rtimes \anel{\Aut}(R).\]
We take a closer look at how the subgroup $\widetilde{S} \leq \Aut(\mbU_4(R))$ is defined. 
In~\cite[pp.~73 and~74]{LevchukOriginal}, 
Levchuk 
observes that a matrix $s = \left( \begin{smallmatrix} a_{1,1} & a_{1,2} \\ a_{2,1} & a_{2,2} \end{smallmatrix} \right) \in \SL_2(R)$ gives rise to an automorphism $\widetilde{s} \in \Aut(\mbU_4(R))$ as long as it satisfies the conditions 
\[ 2a_{1,1}a_{1,2} = 2a_{2,1}a_{2,2} = 0 \, \, \text{ and } \, \, a_{i,1}a_{i,2}(x^2-x)(y^2-y)=0 \]
for $i\in\{1,2\}$ and $x,y\in R$. 
More precisely, the map $\widetilde{s} \in \Aut(\mbU_4(R))$ induced by such an $s = \left( \begin{smallmatrix} a_{1,1} & a_{1,2} \\ a_{2,1} & a_{2,2} \end{smallmatrix} \right)$ is defined via the rules 
\[\widetilde{s}(e_{1,2}(r)) = e_{1,2}(a_{1,1}r) e_{3,4}(a_{1,2}r), \quad \widetilde{s}(e_{3,4}(r)) = e_{1,2}(a_{2,1}r) e_{3,4}(a_{2,2}r),\]
\[\widetilde{s}(e_{2,3}(r)) = e_{2,3}(r) e_{1,3}(a_{1,1}a_{2,1}(r^2-r))e_{2,4}(a_{1,2}a_{2,2}(r^2-r)), \] 
\[\widetilde{s}(e_{1,3}(r)) = e_{1,3}(a_{1,1}r) e_{2,4}(-a_{1,2}r)e_{1,4}(a_{1,1}a_{1,2}r^2),\]
\[\widetilde{s}(e_{2,4}(r)) = e_{1,3}(-a_{2,1}r) e_{2,4}(a_{2,2}r)e_{1,4}(a_{2,1}a_{2,2}r^2), \quad \text{ and}\] 
\[\widetilde{s}(e_{1,4}(r)) = e_{1,4}((a_{1,1}a_{2,2}+a_{1,2}a_{2,1})r). \]
Recalling that $R$ is an integral domain, note that if $\mathrm{char}(R) \neq 2$ the first requirement on such an $s = (a_{i,j}) \in \SL_2(R)$ implies that it must have the form 
\begin{equation} \label{eq:s-Automorphisms}
s = \left( \begin{smallmatrix} a_{1,1} & 0 \\ 0 & a_{2,2} \end{smallmatrix} \right) = \left( \begin{smallmatrix} a_{1,1} & 0 \\ 0 & a_{1,1}^{-1} \end{smallmatrix} \right)  \, \text{ or } \, s = \left( \begin{smallmatrix} 0 & a_{1,2} \\ -a_{2,1} & 0 \end{smallmatrix} \right) = \left( \begin{smallmatrix} 0 & a_{1,2} \\ -a_{1,2}^{-1} & 0 \end{smallmatrix} \right).
\end{equation}
If the integral domain $R$ has $\mathrm{char}(R)=2$ and $R \neq \F_2$, the second requirement on the entries of $s = (a_{i,j}) \in \SL_2(R)$ also implies that $a_{1,1}a_{1,2} = 0$ and $a_{2,1}a_{2,2}=0$. 

Thus, in our case, only matrices as in equation~\bref{eq:s-Automorphisms} produce such automorphisms $\widetilde{s} \in \Aut(\mbU_4(R))$. Going further, the automorphism $\widetilde{s}$ induced by $s = (a_{i,j})$ acts on the generators $e_{k,\ell}(r) \in \mbU_4(R)$ by the following rules. 
\[
\widetilde{s} : 
\begin{cases} e_{1,2}(r) & \mapsto e_{1,2}(a_{1,1}r) \cdot e_{3,4}(a_{1,2}r), \\ 
e_{2,3}(r) & \mapsto e_{2,3}(r), \\ 
e_{3,4}(r) & \mapsto e_{1,2}(a_{2,1}r) \cdot e_{2,2}(a_{2,2}r), \\
e_{1,3}(r) & \mapsto e_{1,3}(a_{1,1}r) \cdot e_{2,4}(-a_{1,2}r), \\
e_{2,4}(r) & \mapsto e_{1,3}(-a_{2,1}r) \cdot e_{2,4}(a_{2,2}r), \\
e_{1,4}(r) & \mapsto e_{1,4}((a_{1,1}a_{2,2}+a_{1,2}a_{2,1})r).
\end{cases}
\]
The subgroup $\widetilde{S} \leq \Aut(\mbU_4(R))$ is precisely the group generated by all automorphisms $\widetilde{s}$ as above~[\emph{loc. cit.}]. 

The point now is that the product $\widetilde{S} \cdot D$ is isomorphic to $\gera{\tau} \ltimes D$, where $\tau$ is the flip automorphism. In effect, since a matrix $s$ giving rise to $\widetilde{s} \in \widetilde{S}$ must be as in condition~\bref{eq:s-Automorphisms}, the action of $\widetilde{s}$ given previously reduces as follows.
\[
\text{For } s=\left( \begin{smallmatrix} a_{1,1} & 0 \\ 0 & a_{1,1}^{-1} \end{smallmatrix} \right) \text{ one has } \, \widetilde{s}: \begin{cases} e_{1,j}(r) & \mapsto e_{1,j}(a_{1,1}r) \, \text{ if } j \in \{ 2,3 \}, \\ 
e_{k,\ell}(r) & \mapsto e_{k,\ell}(r) \, \text{ if } (k,\ell) \in \{ (2,3), (1,4)\}, \\ 
e_{i,4}(r) & \mapsto e_{i,4}(a_{1,1}^{-1}r) \, \text{ if } i \in \{ 2,3 \},
\end{cases}
\]
\[
\text{whereas for } \, s=\left( \begin{smallmatrix} 0 & a_{1,2} \\ -a_{1,2}^{-1} & 0 \end{smallmatrix} \right) \, \text{ one has } \, \widetilde{s} : 
\begin{cases} e_{1,2}(r) & \mapsto e_{3,4}(a_{1,2}r), \\ 
e_{2,3}(r) & \mapsto e_{2,3}(r), \\ 
e_{3,4}(r) & \mapsto e_{1,2}(-a_{1,2}^{-1}r), \\
e_{1,3}(r) & \mapsto e_{2,4}(-a_{1,2}r), \\
e_{2,4}(r) & \mapsto e_{1,3}(a_{1,2}^{-1}r), \\
e_{1,4}(r) & \mapsto e_{1,4}(-r).
\end{cases}
\]
Using the above descriptions, direct matrix computations show that the map $\widetilde{s}$ is in fact a composition of the flip automorphism and conjugation by a diagonal matrix. More precisely,  
\[\widetilde{s} = \tau \circ \iota_d,\]
where $\iota_d \in D$ is conjugation by the diagonal matrix 
\[ d = 
\begin{cases} \left( \begin{smallmatrix} a_{1,1} & & & \\ & 1 & & \\ & & 1 & \\ & & & a_{1,1} \end{smallmatrix} \right), & \text{ in case } s = \left( \begin{smallmatrix} a_{1,1} & 0 \\ 0 & a_{1,1}^{-1} \end{smallmatrix} \right), \text{ or } \\
\left( \begin{smallmatrix} a_{1,2} & & & \\ & 1 & & \\ & & 1 & \\ & & & -a_{1,2} \end{smallmatrix} \right) & \text{ in case } s = \left( \begin{smallmatrix} 0 & a_{1,2} \\ -a_{1,2}^{-1} & 0 \end{smallmatrix} \right).
\end{cases}
\]
Thus the only generator of $\widetilde{S} \cdot D$ not lying in $D$ is the flip automorphism. One readily checks that $D$ is normal in $\widetilde{S} \cdot D$, and that the product splits as $\gera{\tau} \ltimes D$. 

Hence, one may proceed similarly to the last steps of the case $n \geq 5$ to obtain the same equality~\bref{eq:Levchuk}. Indeed,  
\begin{align*}
    \Aut(\mbU_4(R)) & = ((\widetilde{S} \cdot D) \ltimes (\mc{Z} \cdot \Inn(\mbU_4(R)) \cdot \mc{U}^{(c)})) \rtimes \anel{\Aut}(R) \\
    & = ((\gera{\tau} \ltimes D) \ltimes (\mc{Z} \cdot \Inn(\mbU_4(R)) \cdot \mc{U}^{(c)})) \rtimes \anel{\Aut}(R) \\
    & = (\gera{\tau} \ltimes (D \ltimes (\mc{Z} \cdot \Inn(\mbU_4(R)) \cdot \mc{U}^{(c)}))) \rtimes \anel{\Aut}(R) \\
    & = ((D \ltimes (\Inn(\mbU_4(R)) \cdot \mc{Z} \cdot \mc{U}^{(c)})) \rtimes \gera{\tau}) \rtimes \anel{\Aut}(R) \\
    & = D \cdot \Inn(\mbU_4(R)) \cdot \mc{Z} \cdot \mc{U}^{(c)} \cdot \gera{\tau} \cdot \anel{\Aut}(R) \\
    & = \Inn(\mbU_4(R)) \cdot D \cdot \mc{Z} \cdot \mc{U}^{(c)} \cdot \gera{\tau} \cdot \anel{\Aut}(R).
\end{align*}
This finishes the proof. \qed

\section{Restatement and proof of Theorem~\ref{thm:Additive}}\label{sec:provadothmB} 
Before recalling our main technical theorem we need a bit more notation. Given a {ring} automorphism $\alpha \in \anel{\Aut}(R)$ we consider the following automorphisms on the underlying additive group $\Addi(R) = (R,+)$. 
Firstly, we let $\addi{\alpha} \in \Aut(\Addi(R))$ denote the same automorphism $\alpha$ now viewed as a {group} automorphism of $\Addi(R)$, i.e., $\addi{\alpha}(r)$ is just $\alpha(r)$ for any $r \in R$. Secondly, in the direct product $\Addi(R) \times \Addi(R)$ of two copies of $(R,+)$ we let $\tau_\alpha \in \Aut(\Addi(R) \times \Addi(R))$ denote the `flip' automorphism induced by~$\alpha$, that is, $\tau_\alpha((r,s)) = (\addi{\alpha}(s),\addi{\alpha}(r))$.

\theoremstyle{plain}
\newtheorem*{thmB}{Theorem~\ref{thm:Additive}}
\begin{thmB} Let $R$ be an integral domain with a finitely generated group of units $\Mult(R)$. Assume that both $R(\addi{\alpha})$ and $R(\tau_\alpha)$ are infinite for all $\alpha \in \anel{\Aut}(R)$, where $\addi{\alpha}$ and $\tau_\alpha$ are as above. Then the groups $\mbB_n(R)$ and $\PB_n(R)$ have property $\Ri$ for all $n \geq 4$.
\end{thmB}

For the remaining of Section~\ref{sec:provadothmB}, we assume that $R$ is as in the statement and $n \geq 4$. Note that the hypotheses force $R \neq \F_2$, so that Levchuk's theorem will be applicable.

Because $\PB_n(R)$ is a characteristic quotient of $\mbB_n(R)$, Lemma~\ref{lem:desempre} shows that $\mbB_n(R)$ will have property $\Ri$ if $\PB_n(R)$ does so. To prove Theorem~\ref{thm:Additive} it thus suffices to show that all $\phee \in \Aut(\PB_n(R))$ satisfy $R(\phee)=\infty$. This analysis will be done in several steps. The first observation is the following. 

\begin{lem}\label{lem:useless}
Let $G$ be a group and let $H \leq G$ be a characteristic subgroup. Suppose there exists a 
subset $A \subseteq \Aut(H)$ 
such that 
$\Aut(H) = p(\Inn(G)) \cdot A$, 
where 
$p(\Inn(G))$ is the image of $\Inn(G)$ under the obvious map $p : \Inn(G) \to \Aut(H)$, $p(\phee) = \phee\vert_{H}$. 
Then any automorphism $\phee \in \Aut(G)$ can be written as a product $\phee = \iota \circ \psi$ where $\iota \in \Inn(G)$ and $p(\psi)=\psi|_H \in A$.
\end{lem}

\begin{proof}
For the moment let us use $\iota_h$ and $\kappa_g$ to represent conjugation by $h \in H$ and by $g \in G$, respectively. Since $H\leq G$, any inner automorphism $\iota_h \in \Inn(H)$ obviously extends to $\kappa_h \in \Inn(G)$ by setting $\kappa_h(g)=hgh^{-1}$. Denoting by $\overline{\Inn(H)} = \{\kappa_h \mid h \in H\} \leq \Aut(G)$, we have $p(\kappa_h) = \iota_h$ so that $p\left(\overline{\Inn(H)}\right) = \Inn(H)$ and $\Inn(H)$ is canonically isomorphic to its copy $\overline{\Inn(H)} \leq \Aut(G)$. 
Note moreover that, since $H\leq G$ is characteristic, given any $\Phi, \Psi \in \Aut(G)$ one has $(\Phi \circ \Psi)|_H = \Phi|_H \circ \Psi|_H$.

Let $\phee \in \Aut(G)$ be given, and choose a set of representatives $T \subseteq \Aut(G)$ for the distinct right cosets of $\Inn(G)$ in $\Aut(G)$, 
so that 
$\phee = \kappa_0 \circ \tau$ for some $\kappa_0 \in \Inn(G)$ 
and $\tau \in T$. Projecting to $\Aut(H)$ we obtain $p(\kappa_0 \circ \tau) = p(\kappa_0) \circ \tau\vert_H$. By hypothesis, the restriction $p(\tau)=\tau|_H$ can be written as 
$\tau\vert_H = p(\kappa_1) \circ \alpha$ where $\kappa_1 \in \Inn(G)$ 
and $\alpha \in A$. Now define 
$\psi = \kappa_1^{-1} \circ \tau \in \Aut(G)$ and $\kappa=\kappa_0 \circ \kappa_1 \in \Inn(G)$. 
One then has
\[\phee = \kappa_0 \circ \tau = (\kappa_0 \circ \kappa_1) \circ (\kappa_1^{-1} \circ \tau) = \kappa \circ \psi\] 
and $p(\psi)=\psi\vert_H = (\kappa_1^{-1} \circ \tau)\vert_H = \kappa_1^{-1}\vert_H \circ \tau\vert_H = p(\kappa_1)^{-1} \circ p(\kappa_1) \circ \alpha = \alpha \in A$, 
as desired.
\end{proof}

Lemma~\ref{lem:useless} will be used in conjunction with Theorem~\ref{thm:Levchuk} to rewrite each $\phee \in \Aut(\PB_n(R))$ as $\phee = \iota \circ \psi$ with $\iota \in \Inn(\PB_n(R))$ and where the restriction $\psi'$ of $\psi \in \Aut(\PB_n(R))$ to the characteristic subgroup $\mbU_n(R)$ has a convenient description without conjugation as one of its factors. 
The next steps, given in Section~\ref{sec:gpA}, shall yield $R(\psi')=\infty$. Finally, we use this equality in Section~\ref{sec:provaThmB} to show that $R(\phee) = \infty$. 

\subsection{A subset of automorphisms of $\mbU_n(R)$}\label{sec:gpA} 
As shown in Theorem~\ref{thm:Levchuk}, one may write 
\[\Aut(\mbU_n(R)) = \Inn(\mbU_n(R)) \cdot D \cdot \mc{Z} \cdot \mc{U}^{(c)} \cdot \gera{\tau} \cdot \anel{\Aut}(R). \] 
In this section, we consider the 
subset $\mc{A} = \mc{Z} \cdot \mc{U}^{(c)} \cdot \gera{\tau} \cdot \anel{\Aut}(R) \subseteq \Aut(\mbU_n(R))$. 
We claim that, under the hypotheses of Theorem~\ref{thm:Additive}, 
\begin{align} \label{eq:psilinha}
R(\psi') = \infty \, \text{ for all } \, \psi' \in \mc{A}.
\end{align}

By definition, an automorphism $\psi' \in \mc{A}$ is of the form $ \psi' = \mc{z} \circ \sigma \circ \tau^\veps \circ \alpha_\ast$ with  $\veps \in \set{0,1}$, where $\mc{z} \in \mc{Z}$, $\sigma \in \mc{U}^{(c)}$, $\alpha_\ast \in \anel{\Aut}(R)$; cf. Section~\ref{sec:Levchuk}.

Since $\gamma_2(\mbU_n(R)) = [\mbU_n(R), \mbU_n(R)]$ is characteristic in $\mbU_n(R)$, the map $\psi'$ induces an automorphism $\barra{\psi'} = \barra{\mc{z} \circ \sigma \circ \tau^\veps \circ \alpha_\ast}$ on the abelianization $\mbU_n(R)^{\mathrm{ab}}$ and moreover $R(\psi') \geq R(\barra{\psi'})$; c.f.\ Lemma~\ref{lem:desempre}. It suffices to check that $R(\barra{\psi'}) = \infty$. 
The following two lemmata simplify this task by assuring that $\barra{\mc{z} \circ \sigma} = \id$, so that $\barra{\psi'} = \barra{\tau^\veps \circ \alpha_\ast}$ and hence $R(\barra{\psi'}) = R(\barra{\tau^\veps \circ \alpha_\ast})$.

\begin{lem} \label{obs:centralautomorphisms2}
For $n \geq 3$, any central automorphism $\mc{z} \in \mc{Z} \leq \Aut(\mbU_n(R))$ induces the identity on the abelianization $\mbU_n(R)^{\mathrm{ab}}$.
\end{lem}

\begin{proof}
In fact a stronger statement holds: $\mc{z}$ induces the identity on each factor $\gamma_k(\mbU_n(R)) / \gamma_{k+1}(\mbU_n(R))$ for $k < n-1$. This is immediate from the definition since $\mc{z}$ acts as multiplication by an element lying in the center $Z(\mbU_n(R)) = \mc{E}_{1,n}(R)$. (Recall that $\mc{z}$ induces an automorphism $\barra{\mc{z}}_k$ of the quotient $\gamma_k(\mbU_n(R)) / \gamma_{k+1}(\mbU_n(R))$ because the lower central series $\gamma_i(\mbU_n(R))$ is a characteristic series.) In particular, $\barra{\mc{z}}_1 = \id$ on the abelianization $\mbU_n(R)^{\mathrm{ab}} = \gamma_1(\mbU_n(R)) / \gamma_{2}(\mbU_n(R)) = \mbU_n(R) / [\mbU_n(R),\mbU_n(R)]$.
\end{proof}

\begin{lem} \label{obs:sigmaautomorphisms}
For $n \geq 3$, any extremal automorphism $\sigma \in \mathcal{U}^{(c)}$ 
induces the identity on the abelianization $\mbU_n(R)^{\mathrm{ab}}$.
\end{lem}

\begin{proof}
This is also straightforward from the definition, for $e_{2,n}(r)$, $e_{1,n-1}(r)$ and $e_{1,n}(r)$ all belong to $\gamma_2(\mbU_n(R)) = [\mbU_n(R),\mbU_n(R)]$ for all $r \in R$, whence any $\sigma \in \mc{U}^{(c)}$ induces a trivial action on $\mbU_n(R)^{\mathrm{ab}} = \mbU_n(R) / [\mbU_n(R), \mbU_n(R)] = \gamma_1(\mbU_n(R)) / \gamma_2(\mbU_n(R))$. 
\end{proof}

In the following step --- Proposition~\ref{obs:flip} --- we construct a $\barra{\tau^\veps \circ \alpha_\ast}$-invariant subgroup $N$ of the abelianization $\mbU_n(R)^{\mathrm{ab}}$ such that the quotient $Q=\mbU_n(R)^{\mathrm{ab}}/N$ is also $\barra{\tau^\veps \circ \alpha_\ast}$-invariant. Then, in Proposition~\ref{pps:autA}, we derive equation~\bref{eq:psilinha} by showing that the induced automorphism $\barra{\barra{\tau^\veps \circ \alpha_\ast}}$ on $Q$ is such that $R\left(\barra{\barra{\tau^\veps \circ \alpha_\ast}}\right)=\infty$, hence $R(\barra{\tau^\veps \circ \alpha_\ast}) 
=\infty$ by Lemma~\ref{lem:desempre}.

\begin{pps}\label{obs:flip} The abelianization $\mbU_n(R)^{\mathrm{ab}}$ has two $\barra{\tau^\veps \circ \alpha_\ast}$-invariant subgroups --- $E_{\mathrm{mid}}$ and $C_{\mathrm{mid}}$ --- that are  group-theoretic complements of one another. In particular, the map 
{$\barra{\tau^\veps \circ \alpha_\ast}$} further induces automorphisms on both these groups.
\end{pps}

\begin{proof}
We start by constructing the `middle' subgroup $E_{\mathrm{mid}}$. 
As a consequence of the known shape of $\mbU_n(R)$ and its lower central series, we freely identify $\mbU_n(R)$ isomorphically with the cartesian product $\prod_{i=1}^{n-1} \mc{E}_{i,i+1}(R)$ in the obvious way throughout this proof. 
Under this identification, an element of $\mbU_n(R)^{\mathrm{ab}} \cong \prod_{i=1}^{n-1} \mc{E}_{i,i+1}(R)$ is a tuple $(e_{1,2}(r_{1,2}), \ldots e_{n-1,n}(r_{n-1,n}))$. 
For simplicity, we write 
\[\mathbf{e}_m(r) := (e_{1,2}(r_{1,2}), \ldots e_{n-1,n}(r_{n-1,n})),\] 
whenever $r_{i,i+1}=0$ for all $i \neq m$ and $r_{m,m+1}=r$.

We define the `middle' subgroup by setting 
\[ E_{\mathrm{mid}} = \langle\{ \mathbf{e}_m(r), \mathbf{e}_{n - m}(r) \mid r \in R \} \rangle \leq \prod_{i=1}^{n-1} \mc{E}_{i,i+1}(R),\]
where $m = \left\lceil \frac{n-1}{2} \right\rceil$ and $\left\lceil\cdot\right\rceil$ denotes the ceiling function. 
That is, $E_{\mathrm{mid}}$ is just the canonical image of the product 
$\mc{E}_{m, m + 1}(R) \cdot \mc{E}_{n - m, n - m + 1}(R) \leq \mbU_n(R)$ 
in the abelianization $\mbU_n(R)^{\mathrm{ab}} \cong \prod_{i=1}^{n-1} \mc{E}_{i,i+1}(R)$.

For example --- and to justify the notation `middle' --- consider the cases $n = 4$ and $n = 5$. In the former, one has $m = \left\lceil \frac{3}{2} \right\rceil  = 2$ but also $n - m = 2$, so that $\mc{E}_{m, m + 1}(R) = \mc{E}_{n - m, n - m + 1}(R) = \mc{E}_{2, 3}(R)$, hence $E_{\mathrm{mid}}$ is isomorphic to the elementary subgroup $\mc{E}_{2,3}(R)$ found in the the middle of the superdiagonal of $\mbU_4(R)$. In case $n = 5$, one has $m = \left\lceil \frac{4}{2} \right\rceil  = 2$ and $n-m=5-2=3$, so that $E_{\mathrm{mid}}$ is the direct product $\mc{E}_{2,3}(R) \times \mc{E}_{3,4}(R)$ of the two `middle' subgroups of the superdiagonal of $\mbU_5(R)$.

It is straightforward that the middle subgroup $E_{\mathrm{mid}}$ has the group $C_{\mathrm{mid}} = \prod_{i=1}^{n-1} E_{i,i+1}$ as (group-theoretic) complement, where 
\[ E_{i,i+1} = \begin{cases} \mc{E}_{i,i+1}(R), & \text{ if } i \notin \set{ m, \, n - m }, \\
\set{1} & \text{ otherwise}. \end{cases} \]

Let us have a closer look at the automorphism $\barra{\tau}$ induced by the flip automorphism $\tau$ on the abelianization $\mbU_n(R)^{\mathrm{ab}} \cong \prod_{i=1}^{n-1} \mc{E}_{i,i+1}(R)$.
We observe that $\tau$ does not alter the signs of the first superdiagonal entries of matrices of $\mbU_n(R)$. It also maps any subgroup $\mc{E}_{i,i+1}(R)$ on the superdiagonal isomorphically onto another subgroup $\mc{E}_{j,j+1}(R)$ on the superdiagonal.

Using these facts, one concludes that $\barra{\tau}$ acts on $\mbU_n(R)^{\mathrm{ab}}$ via 
\begin{align} \label{tau1} \begin{split} \barra{\tau}( (e_{1,2}(r_1), \, e_{2,3}(r_2), \ldots, \, e_{n-2,n-1}(r_{n-2}), \, e_{n-1,n}(r_{n-1}) )) & = \\
 (e_{1,2}(r_{n-1}), \, e_{2,3}(r_{n-2}), \ldots, \, e_{n-2,n-1}(r_{2}), \, e_{n-1,n}(r_{1}) ) & \end{split} \end{align}
 flipping coordinates, with $\mbU_n(R)^{\mathrm{ab}}$ identified as before as the direct product $\mc{E}_{1,2}(R) \times \mc{E}_{2,3}(R) \times \cdots \times \mc{E}_{n-1,n}(R)$.

The action~\bref{tau1} of the flip automorphism $\tau$ thus shows that both $E_{\mathrm{mid}}$ and $C_{\mathrm{mid}}$ are $\barra{\tau}$-invariant. The fact that both $E_{\mathrm{mid}}$ and $C_{\mathrm{mid}}$ are $\overline{\alpha_\ast}$-invariant follows from Remark~\ref{obs:ringautomorphisms}.
\end{proof}

\begin{pps}\label{pps:autA} Under the hypotheses of Theorem~\ref{thm:Additive}, each automorphism $\psi' \in \mc{A}$ satisfies $R(\psi')=\infty$. 
\end{pps}
\begin{proof}
Write $ \psi' = \mc{z} \circ \sigma \circ \tau^\veps \circ \alpha_\ast \in \mc{A}$ with  $\veps \in \set{0,1}$, where $\mc{z} \in \mc{Z}$, $\sigma \in \mc{U}^{(c)}$, $\alpha_\ast \in \anel{\Aut}(R)$.  
 As shown in Lemmata~\ref{obs:centralautomorphisms2} and~\ref{obs:sigmaautomorphisms}, $\barra{\mc{z} \circ \sigma} = \id$, whence $\barra{\psi'} = \barra{\tau^\veps \circ \alpha_\ast}$.

Following Proposition~\ref{obs:flip}, consider the `middle' subgroup $E_{\mathrm{mid}} \leq \mbU_n(R)^{\mathrm{ab}}$, i.e., the image in $\mbU_n(R)^{\mathrm{ab}}$ of the product 
\[\mc{E}_{m, m + 1}(R) \cdot \mc{E}_{n-m, n - m + 1}(R) \leq \mbU_n(R),\] 
where $m = \left\lceil \frac{n-1}{2} \right\rceil$.  
Since the complement $C_{\mathrm{mid}} \leq \mbU_n(R)^{\mathrm{ab}}$ is also $\barra{\tau^\veps \circ \alpha_\ast}$-invariant, modding it out yields an induced automorphism $\barra{\barra{\tau^\veps \circ \alpha_\ast}}$ on $E_{\mathrm{mid}}$ and $R(\barra{\tau^\veps \circ \alpha_\ast}) \geq R\left( \barra{\barra{\tau^\veps \circ \alpha_\ast}} \right)$ by Lemma~\ref{lem:desempre}. 

Now if $n$ is even then $m=\left\lceil \frac{n-1}{2} \right\rceil = n - \left\lceil \frac{n-1}{2} \right\rceil=n-m$, in which case the middle subgroup $E_{\mathrm{mid}}$ consists of a single copy $\mc{E}_{m, m + 1}(R)$. That is, 
\[ E_{\mathrm{mid}} = \mc{E}_{m, m + 1}(R) \cong \Addi(R), \] 
and moreover $\barra{\barra{\tau^\veps}} = \id$. In case $n$ is odd, one has $n-m=n-\left\lceil \frac{n-1}{2} \right\rceil = n - \frac{n-1}{2} = \frac{n-1}{2}+1 = \left\lceil \frac{n-1}{2} \right\rceil + 1=m+1$. Thus, $E_{\mathrm{mid}}$ is the direct product
\[ E_{\mathrm{mid}} \cong \mc{E}_{m, m + 1}(R) \times \mc{E}_{m + 1, m + 2}(R) \cong \Addi(R) \times \Addi(R). \]
Using the identification above, if $\veps = 1$ it follows from the action~\bref{tau1} shown in Section~\ref{Levchuk:flip} that $\barra{\barra{\tau^\veps}} = \barra{\barra{\tau}}$ acts on $E_{\mathrm{mid}} \cong \Addi(R) \times \Addi(R)$ by the coordinate flip 
\[
\xymatrix@R=2mm{
\barra{\barra{\tau^\veps}} : \Addi(R) \times \Addi(R) \ar[r] & \Addi(R) \times \Addi(R) \\
(r,s) \ar@{|->}[r] & (s,r).
}
\]

Writing $\alpha \in \anel{\Aut}(R)$ for the ring automorphism which induces $\alpha_\ast$, it follows from the above discussion that the induced automorphism $\barra{\barra{\tau^\veps \circ \alpha_\ast}}$ acts on the middle subgroup $E_{\mathrm{mid}}$ as follows.
\[ \barra{\barra{\tau^\veps \circ \alpha_\ast}}(r) = \addi{\alpha}(r) \, \text{ if } E_{\mathrm{mid}} \cong \Addi(R) \, \text{ (case } n \text{ even), otherwise} \]
\[ \barra{\barra{\tau^\veps \circ \alpha_\ast}}(r,s) = \begin{cases} (\addi{\alpha}(r), \addi{\alpha}(s)) & \text{ if } \veps = 0, \\ (\addi{\alpha}(s), \addi{\alpha}(r)) & \text{ if } \veps=1, \end{cases} \text{ if } E_{\mathrm{mid}} \cong \Addi(R) \times \Addi(R).  \]
In the first two cases above one has $R(\barra{\barra{\tau^\veps \circ \alpha_\ast}}) \geq R(\addi{\alpha})$, whereas in the third case ($n$ odd and $\veps=1$) it even holds $\barra{\barra{\tau^\veps \circ \alpha_\ast}} = \tau_\alpha$, where $\addi{\alpha}$ and $\tau_\alpha$ are the automorphisms defined for the statement of Theorem~\ref{thm:Additive}. Altogether, the hypotheses of Theorem~\ref{thm:Additive} yield $R\left(\barra{\barra{\tau^\veps \circ \alpha_\ast}}\right) = \infty$, whence the equality $R(\barra{\psi'}) = \infty$, as desired.    
\end{proof}

\subsection{Concluding the proof of Theorem~\ref{thm:Additive}}\label{sec:provaThmB}

As explained, it suffices to show that $\PB_n(R)$ has property~$\Ri$. Let $\phee \in \Aut(\PB_n(R))$.

Suppose first that $R^\times=\set{1}$. Then $\PB_n(R) \cong \mbU_n(R)$ as the diagonal part of $\PB_n(R)$ is trivial. Moreover the subgroup $D \leq \Aut(\mbU_n(R))$ of automorphisms given by conjugation by a diagonal matrix is also trivial. It follows from Theorem~\ref{thm:Levchuk} that an arbitrary automorphism $\phee \in \Aut(\PB_n(R)) \cong \Aut(\mbU_n(R))$ is of the form $\phee = \iota \circ \psi$ where $\iota \in \Inn(\mbU_n(R))$ and $\psi \in \mc{A} \subseteq \Aut(\mbU_n(R))$ with $\mc{A}$ as in Section~\ref{sec:gpA}. Thus 
{Proposition~\ref{pps:autA}} and Lemma~\ref{lem:ignoreinner} give $R(\phee)=R(\psi) = \infty$. 

Assume now that the integral domain $R$ has at least two units. 
Apply Lemma~\ref{lem:useless} taking $G = \PB_n(R)$, $H = \mbU_n(R)$ and $A = \mc{A}$ as in Section~\ref{sec:gpA}. Again by Lemma~\ref{lem:ignoreinner} we have $R(\phee) = R(\iota \circ \psi) = R(\psi)$. It thus suffices to show that $R(\psi) = \infty$.

Recall that $\psi \in \Aut(\PB_n(R))$ induces an automorphism $\barra{\psi} \in \Aut\left(\frac{\mbD_n(R)}{Z_n(R)}\right)$ on the diagonal and $\psi' := \psi|_{\mbU_n(R)} \in \mc{A} \subseteq \Aut(\mbU_n(R))$. If $R(\barra{\psi}) = \infty$, Lemma~\ref{lem:desempre} already yields $R(\psi) = \infty$. Assume otherwise that $R(\barra{\psi}) < \infty$. Because $\Mult(R)$ and thus $\frac{\mbD_n(R)}{Z_n(R)}$ is finitely generated, we have that $\barra{\psi}$ has finitely many fixed points by Lemma~\ref{lem:ReidemeisterAbelian}. Since $R(\psi') = \infty$ by 
{Proposition~\ref{pps:autA}} it follows from Lemma~\ref{lem:desempre} that $R(\psi) = \infty$ as well, which finishes off the proof. \qed

\begin{rmk} \label{teoremadoTimur}
The reader familiar with results on property $\Ri$ for nilpotent groups might recall a similar theorem due to T.~Nasybullov, namely~\cite[Theorem~1]{TimurUniTri}. It is stated that, if $I$ is an infinite integral domain of characteristic zero which is {finitely generated} as a $\Z$-module, then the unipotent group $\mbU_n(I)$ has property~$\Ri$ whenever $n > 2|I^\times|$. On the other hand, for any integral domain $R$ with $|R^\times| = k < \infty$, the diagonal part $\mbD_n(R)/Z_n(R)$ is also finite, so that Lemma~\ref{lem:desempre}(iii) holds trivially for the exact sequence $\mbU_n(R) \into \PB_n(R) \onto \mbD_n(R)/Z_n(R)$. Thus Nasybullov's Theorem implies the following special cases of our Theorem~\ref{thm:Additive}: for an integral domain $I$ with $|I^\times|=k<\infty$ such that $\Addi(I)$ is free abelian, the groups $\PB_{2k+\ell}(I)$ and $\mbB_{2k+\ell}(I)$ have property~$R_\infty$ for all $\ell \geq 1$. 
For example, taking $I=\Z[i]$, it was already known by Nasybullov's result that $\PB_{m}(\Z[i])$ and $\mbB_m(\Z[i])$ have $R_\infty$ whenever $m \geq 9$, though groups such as $\PB_n(\Z[\sqrt{2}])$ are not covered by his findings as the base ring has infinitely many units.

We point out that there is a small gap in the proof of \cite[Theorem~1]{TimurUniTri}. 
Indeed, Nasybullov's theorem uses \cite[Proposition~7]{TimurUniTri}, which in turn is cited as a version of 
Levchuk's 
theorem for the case of integral domains and $n \geq 3$. However, as stated, \cite[Proposition~7]{TimurUniTri} does \emph{not} include 
extremal automorphisms 
--- that is, the subgroup $\mc{U}^{(c)} \leq \Aut(\mbU_n(R))$ was not considered in~\cite{TimurUniTri}. The proposition thus cannot be applied as stated since there exist integral domains $R$ for which the subgroup $\mc{U}^{(c)} \leq \Aut(\mbU_n(R))$ is non-trivial --- take e.g., $R = \Z[1/2]$. Moreover when $n \leq 4$ the generators of $\Aut(\mbU_n(R))$ also include automorphisms not considered in~\cite{TimurUniTri}; cf. \cite[Theorems~2 and~3 and Corollary~5]{LevchukOriginal}.

Fortunately, the above mentioned gap in~\cite[Theorem~1]{TimurUniTri} can be overcome. First of all the assumptions on $R$ include having characteristic zero and thus $n \geq 5$ since $R$ contains a copy of $\Z$. Secondly, the omitted 
extremal automorphisms 
act trivially on almost all quotients of the lower central series of $\mbU_n(R)$. More precisely, any $\sigma \in \mc{U}^{(c)}$ induces the identity on $\mbU_n(R) / \gamma_{n-2}(\mbU_n(R))$, as can be seen from the description of $\mc{U}^{(c)}$; cf. Lemma~\ref{obs:sigmaautomorphisms}. The remaining arguments in the proof of \cite[Theorem~1, pages 258--261]{TimurUniTri} thus carry over with appropriate modifications. We take the opportunity to thank T.~Nasybullov for promptly discussing his results with us. 
\end{rmk}

\section{Applications of Theorem~\ref{thm:Additive}}\label{applications}

Here we exhibit new families of soluble matrix groups having property~$\Ri$ by applying Theorem~\ref{thm:Additive}. We keep the notation of the previous section for the automorphisms {$\addi{\alpha}$ and $\tau_\alpha$.}

\begin{pps} \label{pps:FptFptt-1OK} 
Let $R_0 \in \set{\Z, \F_p \mid p \in \N \text{ a prime}}$ and suppose $R \in \set{R_0[t], \laurent{R_0}, \ri_\K}$, where $\ri_\K$ is the ring of integers of an {arbitrary} algebraic number field $\K$. Then $R(\addi{\alpha}) = \infty = R(\tau_\alpha)$ for any ring automorphism $\alpha \in \anel{\Aut}(R)$. In particular, the groups $\mbB_n(R)$ and $\PB_n(R)$ have property~$\Ri$ when $n \geq 4$.
\end{pps}

We split the proof in Sections~\ref{sec:Fpt}, \ref{sec:RinftyFptt-1} and~\ref{sec:OK} below. 
This result is a major step towards Theorem~\ref{thm:Lieapplication}, covering the cases where $\RSpec(\Gamma_{n,p}) = \{\infty\}$ for $ n \geq 4$. 

\subsection{First case: polynomials in one variable \texorpdfstring{$R = R_0[t]$}{R = R0[t]}} \label{sec:Fpt} In this section we work with the polynomial rings $R = R_0[t]$ in one variable with coefficients in $R_0 \in \set{\Z, \F_p \mid p \text{ a prime}}$. 
It is a simple exercise to verify that every $\alpha \in \anel{\Aut}(R_0[t])$ is of the form $\alpha(\sum_{i=0}^d f_i t^i) = \sum_{i=0}^d f_i (at+b)^i$ for some $a \in R_0^\times$ and $b \in R_0$. 

We start with the case $R_0 = \F_p$. From the above description, the group $\anel{\Aut}(\F_p[t])$ is finite. We show that $R(\addi{\alpha}) = \infty$ drawing from ideas due to Jabara~\cite{Jabara} and Mitra--Sankaran~\cite{MitraSankaran}. 

Consider the element $s \in R$ given by 
\begin{equation}
s := \prod_{\sigma \in \anel{\Aut}(R)} \sigma(t), \label{eq:Oelemento}
\end{equation}
which can be defined since $R = \F_p[t]$ has finitely many ring automorphisms. As $t \in \F_p[t]$ is not a unit, $s$ is also a non-unit. In particular, we can define proper ideals 
\[ I_n := (s^n) \nsgp R \text{ for } n \in \N, \]
noting that each quotient $R/I_n$ is a finite-dimensional $\F_p$-vector space. 

By construction the element $s \in R$ above is a fixed point of any $\alpha \in \anel{\Aut}(R)$, whence each ideal $I_n$ is $\alpha$-invariant. Thus $\alpha \in \anel{\Aut}(R)$ induces, for every $n \in \N$, a ring automorphism 
\[ \alpha_n \in \anel{\Aut}(R/I_n) \quad \text{ given by } \quad \alpha_n(f+I_n) := \alpha(f) + I_n. \]
Viewing $\alpha$ and $\alpha_n$ as additive automorphisms of $R$ and $R/I_n$, respectively, it follows from Lemma~\ref{lem:desempre} that 
\[R(\addi{\alpha}) \geq R(\addi{(\alpha_n)}) \quad \text{ for every } n \in \N.\]
In what follows we argue that the sequence of natural numbers $(R(\addi{(\alpha_n)}))_{n \in \N}$ admits a strictly increasing subsequence, which therefore implies that $R(\addi{\alpha}) = \infty$

To do so we invoke \cite[Theorem~4.1]{KarelSam} which guarantees that, for an $\F_p$-vector space $V$ and an $\F_p$-linear isomorphism $\Phi : V \to V$, one has 
\[R(\addi{\Phi}) = p^{\dim_{\F_p}(\mathrm{coker}(\Phi - \id))}, \, \text{ where } \, \mathrm{coker}(\Phi - \id) := \frac{V}{\im(\Phi-\id)}\]
and $\addi{\Phi}$ is just the same map viewed as an automorphism of $(V,+)$. In our case we can take $V = R/I_n$ and $\Phi = \alpha_n$ noting that $\alpha_n$ is $\F_p$-linear. Since $R/I_n$ is finite-dimensional, we obtain 
\begin{align*}
\dim_{\F_p}(\ker(\alpha_n-\id)) & = \dim_{\F_p}(R/I_n) - \dim_{\F_p}(\im(\alpha_n-\id)) \\
& =\dim_{\F_p}(\mathrm{coker}(\alpha_n-\id)),
\end{align*}
whence 
\[R(\addi{(\alpha_n)}) = p^{\dim_{\F_p}(\ker(\alpha_n-\id))} = p^{\dim_{\F_p}(\mathrm{Fix}(\alpha_n))}.\] 
The above thus shows that $R(\addi{(\alpha_n)})$ grows with $n$ in case the number of fixed points of $\alpha_n$ increases with $n$. 

\begin{lem} \label{6.1newarg}
For every given $n \in \N$ there exists an $N \in \N$ with $N > n$ and such that $|\mathrm{Fix}(\alpha_n)| < |\mathrm{Fix}(\alpha_N)|$. 
\end{lem}

\begin{proof}
Given $n$, write $d = \dim_{\F_p}(R/I_n)$ and choose $N > p^d + 1 = |R/I_n| + 1$. Then, for every $i,j = 0$, $\ldots$, $p^d$, the elements $s^i + I_N \in R/I_N$ are non-trivial fixed points of $\alpha_N$ and $s^i + I_N \neq s^j + I_N$ if $i \neq j$. Thus 
$|\mathrm{Fix}(\alpha_N)| \geq p^d + 1 > |R/I_n| \geq |\mathrm{Fix}(\alpha_n)|$. 
\end{proof}

Lemma~\ref{6.1newarg} thus implies that
\[R(\addi{\alpha}) \geq \mathrm{sup}\{ R(\addi{(\alpha_n)}) \mid n\in \N\} = \infty,\]
as desired.

Showing that $R(\tau_\alpha) = \infty$ is similar to the previous case. Indeed, 
the direct product of ideals $I_n \times I_n \leq R \times R$ is $\tau_\alpha$-invariant and is the kernel of the natural projection $R \times R \onto R/I_n \times R/I_n$. One analogously defines
\[ \xymatrix@R=2mm{ 
\displaystyle \tau_{\alpha_n} : \frac{R}{I_n} \times \frac{R}{I_n} \ar[r] & \displaystyle \frac{R}{I_n} \times \frac{R}{I_n} \\
(f+I_n, \, g+I_n) \ar@{|->}[r] & (\alpha(g)+I_n, \, \alpha(f)+I_n)
}
 \]
and $R(\tau_\alpha) \geq R(\tau_{\alpha_n})$ holds for all $n\in\N$. Since the elements $(s^i+I_n, s^i+I_n)$ are non-trivial, pairwise distinct elements of $\mathrm{Fix}(\tau_{\alpha_n})$ for every $i < n$, one concludes as in the preceding paragraphs that $\sup_{n}R(\tau_{\alpha_n}) = \infty$. This proves the proposition for $R = R_0[t]$ with $R_0=\F_p$. 

The case $R_0 = \Z$ follows from the previous ones. Take, for instance, the canonical projection $\pi : \Z \onto \Z / 2\Z = \F_2$. Given $\alpha \in \anel{\Aut}(\Z[t])$, which acts via $\alpha(\sum_{i=0}^d f_i t^i) = \sum_{i=0}^d f_i (at+b)^i$, it is clear that the ideal $2\Z[t]$ of polynomials with even coefficients is $\addi{\alpha}$-invariant. Thus $\addi{\alpha}$ induces the automorphism $\barra{\addi{\alpha}}$ on $\Z[t] / 2\Z[t] \cong \F_2[t]$ given by $\addi{\alpha}(\sum_{i=0}^d \pi(f_i) t^i) = \sum_{i=0}^d \pi(f_i) (\pi(a)t + \pi(b))^i$. But $R(\addi{\alpha}) \geq R(\barra{\addi{\alpha}})$ by Lemma~\ref{lem:desempre}, and we have just shown that such an $\barra{\addi{\alpha}}$ on $\F_2[t]$ has infinite Reidemeister number. Similarly, the product of ideals $2\Z[t] \times 2\Z[t]$ is also $\tau_\alpha$-invariant, whence taking the induced automorphism $\barra{\tau}_\alpha$ on $\F_2[t] \times \F_2[t]$ yields $R(\tau_\alpha) \geq R(\barra{\tau}_\alpha) = \infty$ by the previous considerations. Applying Theorem~\ref{thm:Additive} finishes off the proof of the present cases. \qed

\subsection{The case of Laurent polynomials \texorpdfstring{$R = \laurent{R_0}$}{R = R0[t,1/t]}} \label{sec:RinftyFptt-1} 
Turning to the second case of Proposition~\ref{pps:FptFptt-1OK} we let $R$ denote the ring of Laurent polynomials $R = R_0[t,t^{-1}]$ throughout this section. Again, $R_0$ is either the ring of integers $\Z$ or a finite field $\F_p$ with a prime number $p$ of elements. We first remark that, in such cases, $R^\times = \{ ut^{n} \mid u \in R_0^{\times}, n \in \Z\}$. 

A straightforward verification shows that the group of ring automorphisms $\anel{\Aut}(\laurent{R_0})$ is isomorphic to $R_0^\times \times C_2$ for any choice of $R_0 \in \{\Z, \F_p \mid p \text{ a prime}\}$. 
To see why, first note that $\alpha\vert_{R_0} = \id\vert_{R_0}$ when $\alpha\in \anel{\Aut}(\laurent{R_0})$, so that $\alpha$ is completely determined by the assignment $t\mapsto \alpha(t)$. As $t$ is a unit in $\laurent{R_0}$, $\alpha(t)$ must equal some $ut^{n}$ with $u \in R_0^\times$ and $n\in \Z$. But if $n \neq \pm 1$, then $\alpha$ would not be surjective. (Recall that $t$ is a torsion-free unit.) 

Thus, given $\alpha \in \anel{\Aut}(\laurent{R_0})$ and an arbitrary Laurent polynomial $f(t) = \sum_{\ell=-\infty}^\infty f_\ell t^\ell \in \laurent{R_0}$, where only finitely many coefficients $f_\ell\in R_0$ are non-zero, we have $\alpha(f(t)) = \sum_{\ell=-\infty}^\infty f_\ell u^\ell t^{\pm \ell}$.

Now let $\mc{o} = \mathrm{ord}(R_0^\times)$.  
We argue that the elements $t^{\mc{o}\ell} \in \laurent{R_0}$ with $\ell \in \N$ yield infinitely many $\addi{\alpha}$-twisted conjugacy classes. Take $i > j \in \N$. Then $t^{\mc{o}i}$ and $t^{\mc{o}j}$ are $\addi{\alpha}$-conjugate if and only if there exists $h(t) = \sum_{\ell=-\infty}^\infty h_\ell t^\ell \in \laurent{R_0}$ such that 
\begin{equation} 
\label{6.2newarg}
0\neq t^{\mc{o}i} - t^{\mc{o}j} = (\id-\addi{\alpha})(h(t)) = \sum_{\ell=-\infty}^\infty h_\ell t^\ell - \sum_{\ell=-\infty}^\infty h_\ell u^\ell t^{\pm\ell}.
\end{equation}
We consider two cases. If $\alpha$ sends $t$ to $ut$ with $u \in R_0^\times$, Eq.~(\ref{6.2newarg}) implies in particular that the Laurent polynomial $h(t)$ must satisfy the condition 
\[ 1 = h_{\mc{o}i} - u^{\mc{o}i}h_{\mc{o}i} = h_{\mc{o}i}(1-(u^{\mc{o}})^{i}) = h_{\mc{o}i}(1-1)=0, \]
which is impossible and hence no such $h(t)$ exists. 
Assume otherwise that $\alpha$ maps $t$ to $ut^{-1}$ with $u \in R_0^\times$. 
Since $i$ is positive and distinct from $j$, in particular, we obtain from Eq.~(\ref{6.2newarg}) the system of equations  
\[ \begin{cases}
 1 = h_{\mc{o}i} - u^{-\mc{o}i}h_{-\mc{o}i} = h_{\mc{o}i} - (u^{\mc{o}})^{-i}h_{-\mc{o}i} = h_{\mc{o}i} - h_{-\mc{o}i} \\
 0= h_{-\mc{o}i} -u^{\mc{o}i}h_{\mc{o}i} = h_{-\mc{o}i} -(u^{\mc{o}})^{i} h_{\mc{o}i} = -h_{\mc{o}i} + h_{-\mc{o}i}
\end{cases}\]
for some coefficients of $h(t)$, which has no solutions over $R_0$. Thus such an element $h(t) \in \laurent{R_0}$ never exists, that is $t^{\mc{o}i}$ and $t^{\mc{o}j}$ can not be $\addi{\alpha}$-conjugate if $i \neq j$, whence $R(\addi{\alpha}) = \infty$. 

We now check that the additive automorphism $\tau_\alpha((r,s)) = (\addi{\alpha}(s),\addi{\alpha}(r))$ of $\laurent{R_0} \times \laurent{R_0}$ also has $R(\tau_\alpha) = \infty$. Similarly to the above we take $\mc{o} = \mathrm{ord}(R_0^\times)$ and show that the elements $(t^{\mc{o}i},0)$, $(0,-t^{\mc{o}j}) \in \laurent{R_0} \times \laurent{R_0}$ with $i > j \in \N$ define infinitely many distinct $\tau_\alpha$-twisted conjugacy classes. 

To begin with, $(t^{\mc{o}i},0)$ and $(0,-t^{\mc{o}j})$ are $\tau_\alpha$-conjugate if and only if there exists a pair of Laurent polynomials $(h(t), g(t)) = (\sum_{\ell=-\infty}^\infty h_\ell t^\ell, \sum_{m=-\infty}^\infty g_m t^m) \in R_0[t,t^{-1}] \times R_0[t,t^{-1}]$ satisfying 
\begin{align} \label{6.3newarg} 
\begin{split}
& (t^{\mc{o}i}, \, \, t^{\mc{o}j}) = (h(t), \, \, g(t)) - \tau_\alpha( \, (h(t), \, g(t)) \, ) \\
& = (\, h(t) - \addi{\alpha}(g(t)), \, \, g(t) - \addi{\alpha}(h(t)) \, ) \\
& = \left( \sum_{\ell=-\infty}^\infty h_\ell t^\ell - \sum_{m=-\infty}^\infty g_m u^m t^{\pm m} , \, \, \sum_{m=-\infty}^\infty g_m t^m - \sum_{\ell=-\infty}^\infty h_\ell u^\ell t^{\pm \ell} \right).
\end{split}
\end{align}
In particular, one extracts from Eq.~\eqref{6.3newarg} the following two possible systems of coefficients. Firstly, if $\alpha$ sends $t$ to $ut$ with $u \in R_0^\times$, we have the system 
\[ \begin{cases}
 1 = h_{\mc{o}i} - g_{\mc{o}i}u^{\mc{o}i} = h_{\mc{o}i}-g_{\mc{o}i}, \\
 0 = g_{\mc{o}i} - h_{\mc{o}i}u^{\mc{o}i} = -h_{\mc{o}i}+g_{\mc{o}i}. 
\end{cases}\]
Secondly, if $\alpha$ sends $t$ to $ut^{-1}$ with $u \in R_0^\times$, the system we obtain is 
\[ \begin{cases}
 1 = h_{\mc{o}i} - g_{-\mc{o}i} u^{-\mc{o}i} = h_{\mc{o}i} - g_{-\mc{o}i}, \\
 0 = g_{-\mc{o}i} - h_{\mc{o}i} u^{\mc{o}i} = -h_{\mc{o}i} + g_{-\mc{o}i}. 
\end{cases}\] 
In both cases, no solutions exist over $R_0$, so that there is no such pair $(h(t), g(t))$ satisfying the above. Therefore $(t^{\mc{o}i},0)$ and $(0,-t^{\mc{o}j})$ lie in distinct $\tau_\alpha$-twist conjugacy classes for all $i > j \in \N$. 

We now argue that, for $i>j \in \N$, the element $(t^{\mc{o}i},0)$ is never $\tau_{\alpha}$-twisted conjugate to $(t^{\mc{o}j},0)$. (By entirely analogous arguments, $(0,-t^{\mc{o}i})$ and $(0,-t^{\mc{o}j})$ also lie in distinct $\tau_{\alpha}$-twisted conjugacy classes whenever $i > j$.) This will conclude the proof that $R(\tau_\alpha)=\infty$. 

By definition, the elements $(t^{\mc{o}i},0)$ and $(t^{\mc{o}j},0)$ are $\tau_\alpha$-conjugate if and only if there exists a pair of Laurent polynomials $(h(t), g(t)) = (\sum_{\ell=-\infty}^\infty h_\ell t^\ell, \sum_{m=-\infty}^\infty g_m t^m) \in R_0[t,t^{-1}] \times R_0[t,t^{-1}]$ satisfying 
\[(t^{\mc{o}i}-t^{\mc{o}j},0)=\left( \sum_{\ell=-\infty}^\infty h_\ell t^\ell - \sum_{m=-\infty}^\infty g_m u^m t^{\pm m} , \, \, \sum_{m=-\infty}^\infty g_m t^m - \sum_{\ell=-\infty}^\infty h_\ell u^\ell t^{\pm \ell} \right).\]
Again, we split in cases $\alpha(t)=ut$ and $\alpha(t)=ut^{-1}$ with $u \in R_0^\times$. Firstly, if $\alpha(t)=ut$, we have the system 
\[ \begin{cases}
 1 = h_{\mc{o}i} - g_{\mc{o}i}u^{\mc{o}i} = h_{\mc{o}i}-g_{\mc{o}i}, \\
 0 = g_{\mc{o}i} - h_{\mc{o}i}u^{\mc{o}i} = -h_{\mc{o}i}+g_{\mc{o}i}. 
\end{cases}\]
Secondly, if $\alpha(t)=ut^{-1}$, the system we obtain is 
\[ \begin{cases}
 1 = h_{\mc{o}i} - g_{-\mc{o}i} u^{-\mc{o}i} = h_{\mc{o}i} - g_{-\mc{o}i}, \\
 0 = g_{-\mc{o}i} - h_{\mc{o}i} u^{\mc{o}i} = -h_{\mc{o}i} + g_{-\mc{o}i}. 
\end{cases}\] 
Either way, no solutions exist over $R_0$, hence there is no such pair $(h(t), g(t))$ satisfying the above. 

Thus $R(\alpha)=R(\tau_\alpha) = \infty$, so that $\Ri$ for $\mbB_n(\laurent{R_0})$ and $\PB_n(\laurent{R_0})$ with $n\geq 4$ is now a consequence of Theorem~\ref{thm:Additive}. \qed

\subsection{The case where \texorpdfstring{$R = \ri_\K$}{R = Ok}, a ring of integers} \label{sec:OK} 
Recall that the integral closure of $\Z$ in a finite extension of $\Q$ is a free $\Z$-module of finite rank, whence the underlying additive group $\Addi(\ri_\K)$ of any ring of integers $\ri_\K$ is finitely generated. 

Suppose $\alpha$ is any ring automorphism of $\ri_\K$. To prove the proposition it suffices to check that $\addi{\alpha}$ and $\tau_\alpha$ have infinitely many fixed points due to Lemma~\ref{lem:ReidemeisterAbelian} and the observation above. This is obviously true for $\addi{\alpha}$ since $\Z \subseteq \ri_\K$ is always contained in the set of fixed points of $\alpha$. And for the automorphism $\tau_\alpha : \ri_\K \times \ri_\K \to \ri_\K \times \ri_\K$, $\tau_\alpha(r,s) = (\addi{\alpha}(s), \, \addi{\alpha}(r))$, the diagonal of $\Z \times \Z$ is clearly contained in the set of fixed points of $\tau_\alpha$. Using Theorem~\ref{thm:Additive} we conclude the proof of Proposition~\ref{pps:FptFptt-1OK}. \qed

\section{Some groups in positive characteristic that do not have property \texorpdfstring{$\Ri$}{R infinity}} \label{sec:newexampleswithoutRinfty}

In contrast with Theorem~\ref{thm:Additive} we now give examples of metabelian $S$-arithmetic groups in positive characteristic not having property~$\Ri$. Our list includes finitely generated and non-finitely generated examples.


\begin{pps}\label{nori} Let $\F_q$ denote the finite field with $q$ elements. 
\begin{enumerate}
	\item \label{nori.i} 
	There exist automorphisms 
	\[\phee_{\mathbb{A}} \in \Aut(\PB_2(\F_q[t])), \,\, \phee_{\mbB} \in \Aut(\mbB_2(\F_q[t])), \,\, \Phi \in \Aut(\mbU_2(\F_q[t]))\]
	with Reidemeister numbers 
    \[ R(\phee_{\mathbb{A}}) = q-1, \quad R(\phee_{\mbB}) = (q-1)^2, \quad R(\Phi) = 1, \]
    respectively. 
	\item \label{nori.ii} 
	If $q \geq 4$, there exist automorphisms $\phee_{A} \in \Aut(\PBplus_2(\F_q[t,t^{-1}]))$, $\phee_{B} \in \Aut(B^+_2(\F_q[t,t^{-1}]))$ and $\phee' \in \Aut(\mbU_2(\F_q[t,t^{-1}]))$ 
	with 
    \[R(\phee_A) = 2, \quad R(\phee_B) = 4, \quad R(\phee') = 1,\] respectively. 
\end{enumerate}
In particular, 	\emph{none} of the groups above have property~$\Ri$.
\end{pps}

The groups $\PBplus_n(R)$ and $B_n^+(R)$ are just variants of $\PB_n(R)$ and $\mbB_n(R)$, respectively, without torsion on the diagonal part; cf. Section~\ref{fqttmo} below for a definition. 

Proposition~\ref{nori} has some overlap with existing results. For $q = p$ a prime, the group $\PBplus_2(\F_p[t,t^{-1}])$ is isomorphic to the lamplighter group $\mc{L}_p = C_p \wr \Z$. 
Using different techniques, Gon\c{c}alves and Wong~\cite{DacibergWongWreath} completely classified which wreath products of the form $A \wr \Z$ have property~$\Ri$, where $A$ is a finitely generated abelian group. In particular, for the small fields $\F_2$, $\F_3$ excluded from Proposition~\dref{nori}{nori.ii}, their results imply that $\PBplus_2(\laurent{\F_2})$ and $\PBplus_2(\laurent{\F_3})$ actually {do have} property~$\Ri$, while $\PBplus_2(\laurent{\F_q})$ does not when $q\geq 4$. 

Combining Propositions~\ref{pps:FptFptt-1OK} and~\ref{nori} --- along with the results of Gon\c{c}alves--Wong on $\PBplus_2(\laurent{\F_2})$ and $\PBplus_2(\laurent{\F_3})$ --- yields Table~\ref{tabela}. 
\begin{table*}[h] 
	\centering
		\begin{tabular}{ c | c | c | c }
				Does $G_n(R)$ have $\Ri$?	& $n=2$ 	& $n=3$ & $n \geq 4$ \\ \hline
				$G_n\in\{\mbB_n, \PB_n\}$ & No 	& Unknown & Yes \\ $R=\F_q[t]$ & & \\ \hline
				$G_n\in \{B_n^+, \PBplus_n\}$ & No 	& Unknown & Unknown \\ $R=\F_q[t]$ & & \\ \hline
				$G_n\in\{\mbB_n, \PB_n\}$ & Unknown  	& Unknown & Yes \\ $R=\F_q[t,t^{-1}]$ &  & \\ \hline
				$G_n\in \{B_n^+, \PBplus_n\}$ & No $\iff q \geq 4$ & Unknown & Unknown \\ $R=\F_q[t,t^{-1}]$ & Yes: $q=2,3$  & \\  \hline
		\end{tabular}
		\caption{Some soluble groups with and without $R_\infty$.} \label{tabela}
\end{table*} 
The $3\times 3$ case, unclear to us, might be an interesting test case for property~$\Ri$. The `plus' versions $B_n^+(\F_q[t])$ and $\PBplus_n(\F_q[t])$ are isomorphic to $\mbU_n(\F_q[t])$ by definition. 
Whether $B_n^+(\laurent{\F_q})$ and $\PBplus_n(\laurent{\F_q})$ have~$\Ri$ for large~$n$, and whether $\mbB_2(\laurent{\F_q})$ and $\PB_2(\laurent{\F_q})$ do not have~$\Ri$, is at the moment unknown to us. Note that our companion result \cite[Theorem~5.1(ii)]{Bn1}, which only deals with finitely presented groups, does not cover these particular $S$-arithmetic rings in positive characteristic.

\begin{rmk}
We stress that there cannot be a version of Theorem~\ref{thm:Additive} in the `plus' case removing torsion from the diagonal --- at least not for all $n \geq 4$. That is, the hypotheses of Theorem~\ref{thm:Additive} might hold while $B_n^+(R)$ and $\PBplus_n(R)$ do not necessarily have~$\Ri$. For example, take $R = \ri_\K$ a ring of integers of a number field $\K$ such that $\ri_\K$ has finitely many units (i.e., $\K$ is either $\Q$ or an imaginary quadratic field). As in Section~\ref{sec:OK}, the maps $\alpha_{\mathrm{add}}$ and $\tau_\alpha$ always have infinite Reidemeister number over $\ri_\K$. On the other hand, $B_n^+(\ri_\K) \cong \PBplus_n(\ri_\K) \cong \mbU_n(\ri_\K)$ and, by \cite[Proposition~8]{TimurUniTri}, it is known that $\RSpec(\mbU_4(\ri_\K)) \neq \{\infty\}$.
\end{rmk}

\subsection{Proof of Proposition~\dref{nori}{nori.i}}\label{fqt}
Throughout this section, $R$ denotes $\F_q[t]$, the polynomial ring in one variable over $\F_q$. 

First note that, because $\F_q$ is not algebraically closed, there exist $d \geq 2$ and a monic polynomial $P(X) = a_0 + a_1 X +  \ldots + a_{d-1} X^{d-1} + X^d$ 
with coefficients in $\F_q$ and \emph{irreducible} over $\F_q$. 
Certainly $a_0 \neq 0$ for otherwise $P(X)$ would be divisible by $X$. The companion matrix 
\[ C_P = \begin{pmatrix} 0 & 0 & \ldots & 0 & -a_0 \\ 1 & 0 & \ldots & 0 & -a_1 \\ 0 & 1 & \ldots & 0 & -a_2 \\ \vdots & \vdots & \ddots & \vdots & \vdots \\ 0 & 0 & \ldots & 1 & -a_{d-1} \end{pmatrix} \]
of $P(X)$ thus has determinant $\det(C_P) = a_0 \neq 0$ and therefore defines an element of $\GL_{d}(\F_q)$. Moreover, the characteristic polynomial of $C_P$ equals $P(X)$ itself. Since $P(X)$ is irreducible, one has {$\det(C_P - \lambda\cdot\mb{1}_d) \neq 0$} for all $\lambda \in \F_q \backslash \set{0}$, hence
\begin{align} \label{eq:Lula}
{\det(\id - a \cdot C_P) \neq 0} 
\text{ for any scalar } a \in \F_q \backslash \set{0}.
\end{align}
That is, the linear map $\id - a \cdot C_P$ still lies in $\GL_d(\F_q)$ if $a \neq 0$.

We would like to define an additive automorphism $\Phi \in \Aut(\Addi(\F_q[t]))$ induced by $C_P$. To do this we first write $\F_q[t]$ as an infinite dimensional $\F_q$-vector space on the standard basis $\set{1, t, t^2, \ldots}$, that is 
\[ \F_q[t] = \bigoplus_{\ell=0}^\infty \F_q \cdot t^\ell. \]
We remark that the elements $x \in \F_q \subset \F_q[t]$ can be viewed both as vectors $xt^0+0t^1+0t^2+\ldots \in \oplus_{\ell=0}^\infty \F_q t^\ell$ and as scalars of the ground field $\F_q$. In particular, an $\F_q$-linear transformation $T : \F_q[t] \to \F_q[t]$ {need not} fix the {vector} $xt^0+0t^1+0t^2+\ldots$ corresponding to the element $x \in \F_q$, though it still satisfies $T(x \cdot v) = x \cdot T(v)$ for any vector $v \in \F_q[t]$ and scalar $x \in \F_q$.

To construct our automorphism $\Phi$ we decompose $\F_q[t]$ into finite-dimensional blocks. More precisely, an arbitrary vector $v \in \F_q[t]$ shall be written as a sum $v = v_1+v_2+v_3+\ldots$ where each $v_k$ lies in the 
$k$-th 
$d$-dimensional block of the space $\oplus_{\ell=0}^\infty \F_q t^\ell$. That is, we decompose 
\[ \F_q[t] = \bigoplus_{\ell=0}^\infty \F_q \cdot t^\ell = \bigoplus_{k=1}^\infty \bigoplus_{\ell=(k-1)d}^{kd-1} \F_q \cdot t^\ell \]
and write $v \in \F_q[t]$ as
\[ v = \sum_{k=0}^\infty v_k \, \text{ where } \, v_k \in \bigoplus_{\ell=(k-1)d}^{kd-1} \F_q \cdot t^\ell \cong \F_q^d. \]
Since $C_P \in \GL_d(\F_q)$, we may define on each $d$-dimensional block  
$\oplus_{\ell=(k-1)d}^{kd-1} \F_q t^\ell \cong \F_q^d$ 
above the (invertible) linear transformation 
 \[ \xymatrix@R=2mm{ 
\displaystyle C_{P,k} : \bigoplus_{\ell=(k-1)d}^{kd-1} \F_q \cdot t^\ell \ar[r] & \displaystyle\bigoplus_{\ell=(k-1)d}^{kd-1} \F_q \cdot t^\ell \\
x \ar@{|->}[r] & C_P(x),
}
 \]
which is just a copy of $C_P$. Thus the map $\Phi : \F_q[t] \to \F_q[t]$ defined via 
\[ 
\xymatrix@R=2mm{ 
\Phi : \displaystyle\bigoplus_{k=1}^\infty \bigoplus_{\ell=(k-1)d}^{kd-1} \F_q \cdot t^\ell \ar[r] & \displaystyle \bigoplus_{k=1}^\infty \bigoplus_{\ell=(k-1)d}^{kd-1} \F_q \cdot t^\ell \\
v = \displaystyle\sum_{k=1}^\infty v_k \ar@{|->}[r] & \displaystyle\sum_{k=1}^\infty C_{P,k}(v_k)
} 
\]
is an automorphism of $\F_q[t]$ as an infinite-dimensional $\F_q$-vector space because each $C_{P,k}$ is an invertible linear transformation of the $d$-dimensional subspace 
$\bigoplus_{\ell=(k-1)d}^{kd-1} \F_q t^\ell$. 
In particular, $\Phi \in \Aut(\Addi(\F_q[t]))$.

Given a unit $a \in \F_q \backslash \set{0}$, let $\mc{m}_a : \Addi(\F_q[t]) \to \Addi(\F_q[t])$ denote the additive automorphism which is just multiplication by $a$. We claim that 
\begin{equation} \label{eq:R(maPhi)}
R(\mc{m}_a \circ \Phi) = 1 \text{ for all } a \in \F_q \backslash \set{0},
\end{equation} 
where $\mc{m}_a$ and $\Phi$ are here viewed as automorphisms of $\Addi(\F_q[t])$. Let then 
$r = \sum_{k=1}^\infty r_k \in \F_q[t]$ 
be arbitrary, where the $r_k$ are as above. Since $\id-a\cdot C_{P,k} \in \GL_d(\F_q)$ by the definition of $C_{P,k}$ and by equality~\bref{eq:Lula}, we may define 
\[ s_k := (\id-a\cdot C_{P,k})^{-1}(r_k) \in \bigoplus_{\ell=(k-1)d}^{kd-1} \F_q \cdot t^\ell \] 
for all $k$ and set 
$s = \sum_{k=1}^\infty s_k \in \F_q[t]$. 
One then has 
\begin{align*} 
r = \sum_{k=1}^\infty r_k = \sum_{k=1}^\infty (\id-a\cdot C_{P,k})(s_k) & = (\id-a\cdot\Phi)\left(\sum_{k=1}^\infty s_k\right) \\ & = s - \mc{m}_a \circ \Phi(s).
\end{align*} 
Thus $r$ is $(\mc{m}_a\circ\Phi)$-conjugated to the zero vector $0 \in \F_q[t]$, whence the claim.

Equality~\bref{eq:R(maPhi)} alone implies that $\mbU_2(\F_q[t]) = \left( \begin{smallmatrix} 1 & \F_q[t] \\ 0 & 1 \end{smallmatrix} \right) \cong \Addi(\F_q[t])$ does not have property~$\Ri$. To verify that the same is true for the groups $\PB_2(\F_q[t])$ and $\mbB_2(\F_q[t])$, consider the maps 
\[
\xymatrix@R=2mm{ \phee_{\mathbb{A}} : \PB_2(\F_q[t]) \ar[r] & \PB_2(\F_q[t]) \\ 
{\begin{bmatrix} u & r \\ 0 & v \end{bmatrix}} \ar@{|->}[r] & {\begin{bmatrix} u & \Phi(r) \\ 0 & v \end{bmatrix}} } 
\]
and 
\[ 
\xymatrix@R=2mm{ \phee_{\mathbf{B}} : \mbB_2(\F_q[t]) \ar[r] & \mbB_2(\F_q[t]) \\ 
{\begin{pmatrix} u & r \\ 0 & v \end{pmatrix}} \ar@{|->}[r] & {\begin{pmatrix} u & \Phi(r) \\ 0 & v \end{pmatrix}}.} 
\]
These are in fact automorphisms because $\Phi : \F_q[t] \to \F_q[t]$ is $\F_q$-linear. By definition, $\phee_{\mathbb{A}}$ and $\phee_{\mathbf{B}}$ preserve the diagonal and unipotent parts of the given groups. 
Recalling that $\mbU_2(\F_q[t])$ is characteristic in the groups above and because their diagonal parts are finite, 
we know from Lemma~\ref{lem:desempre} that $R(\phee_{\mathbb{A}})$ and $R(\phee_{\mathbf{B}})$ are finite in case $R(\iota_{\mathbb{A}} \circ \phee_{\mathbb{A}}')$ and $R(\iota_{\mathbf{B}} \circ \phee_{\mathbf{B}}')$ are also finite for all $\iota_{\mathbb{A}} \in \Inn(\PB_2(\F_q[t]))$ and $\iota_{\mbB} \in \Inn(\mbB_2(\F_q[t]))$. Here, $\phee_{\mathbb{A}}'$ and $\phee_{\mbB}'$ denote the restrictions to $\mbU_2(\F_q[t]) \cong \Addi(\F_q[t])$ of $\phee_{\mathbb{A}}$ and $\phee_{\mbB}$, respectively. But direct matrix computations show that $\iota_{\mathbb{A}} \circ \phee_{\mathbb{A}}'$ and $\iota_{\mathbf{B}} \circ \phee_{\mathbf{B}}'$ are of the form $\mc{m}_a \circ \Phi$ and $\mc{m}_b \circ \Phi$, respectively, for some units $a,b \in \F_q \backslash \set{0}$. Since $R(\mc{m}_a \circ \Phi) = R(\mc{m}_b \circ \Phi) = 1$ by equality~\bref{eq:R(maPhi)}, it follows that $\phee_{\mathbb{A}}$ and $\phee_{\mbB}$ have finite Reidemeister number as well. Therefore $\PB_2(\F_q[t])$ and $\mbB_2(\F_q[t])$ do not have property $\Ri$.

Lastly, the Reidemeister numbers of the maps above can be promptly computed. Firstly, $R(\Phi)=R(\mc{m}_1 \circ \Phi) = 1$. Now, since $\phee_{\mathbb{A}}$ and $\phee_{\mbB}$ fix the diagonal pointwise, every diagonal matrix lies in a distinct Reidemeister class, so that $R(\phee_{\mathbb{A}}) \geq q-1$ and $R(\phee_{\mbB}) \geq (q-1)^2$. Since every $r \in \F_q[t]$ can be written as $r = s - u \cdot \Phi(s)$ for any unit $u \in \F_q \backslash \{0\}$ and for some $s \in \F_q[t]$, it follows that an arbitrary element $\left( \begin{smallmatrix} u & r \\ 0 & v\end{smallmatrix} \right) \in \mbB_2(\F_q[t])$ is $\phee_{\mbB}$-conjugate to its diagonal part $\left( \begin{smallmatrix} u & 0 \\ 0 & v\end{smallmatrix} \right)$ because 
\[ \begin{pmatrix} u & r \\ 0 & v \end{pmatrix} =\begin{pmatrix} u & s - u \cdot \Phi(s) \\ 0 & v \end{pmatrix} = \begin{pmatrix} 1 & s \\ 0 & v \end{pmatrix} \cdot \begin{pmatrix} u & 0 \\ 0 & v \end{pmatrix} \cdot \begin{pmatrix} 1 & \Phi(s) \\ 0 & v \end{pmatrix}^{-1}.\] 
This means that 
\[\left[\begin{pmatrix}	u & 0 \\ 0 & v \end{pmatrix}\right]_{\varphi_{\mathbf{B}}}= \left\{\begin{pmatrix} u & r \\ 0 & v \\ \end{pmatrix}\mid r \in \F_q[t] \right\}.
\]
(Analogously for $ \PB_2(\F_q[t])$.)  Thus $R(\phee_{\mathbb{A}}) = q-1$ and $R(\phee_{\mbB}) = (q-1)^2$, as claimed. \qed

\subsection{Proof of Proposition~\dref{nori}{nori.ii}} \label{fqttmo}
Throughout this section, $R$ denotes the Laurent polynomial ring $\F_q[t,t^{-1}]$ with $q\geq 4$ a power of a prime. 

The group of units of $R$ is 
\[R^{\times}=\{u t^k \mid u \in \F_q^\times, \, k \in \Z\},\] 
which has torsion subgroup $t(R^\times)=\F_{q}^{\times}$. The group $R^\times_{\mathrm{tf}}=\{t^k\mid k \in \Z\}$ is a torsion-free complement of $t(R^\times)$, and we define 
\begin{alignat*}{3} B_{n}^{+}(R)&:= \mbB_n(R ; R^\times_{\mathrm{tf}}) &&:= \left\{ b=(b_{i,j}) \in \mbB_n(R) \bigm| b_{i,i} \in R^\times_{\mathrm{tf}},\, b_{i,j} \in R \right\},\\	
\PBplus_n(R) &:= \PB_n(R ; R^\times_{\mathrm{tf}}) &&:= \left\{ b=[b_{i,j}] \in \PB_n(R) \bigm| b_{i,i} \in R^\times_{\mathrm{tf}},\, b_{i,j} \in R \right\}.
\end{alignat*}
In other words, $B_n^+(R)\leq \mbB_n(R)$ and $\PBplus_n(R) \leq \PB_n(R)$ are the subgroups whose non-trivial diagonal entries are torsion-free units. (These subgroups are not unique setwise since they depend on the choice of $R^\times_{\mathrm{tf}}$, but they are unique up to isomorphism.)

\begin{lem}\label{lemmaa} Let $\F_q$ denote a finite field with $q \geq 4$ elements. There exists a non-zero $a \in \F_q$ such that $1-a^2$ is a unit of $\F_q$. In particular, $1-a$ is a unit of $\F_q$, and the system
\begin{align*} X-aY&=\ell\\
			  -aX+Y&=m
\end{align*}
has a solution $(X,Y) \in \F_{q}^{2}$ for all choices of $\ell, m \in \F_q$.
\end{lem}
\begin{proof} Take $a$ to be a generator of $\F_{q}^{\times}$. Since $|\F_q|\geq 4$, it follows that $a^2\neq 1$.  
\end{proof}

Let $a\in \F_{q}^{\times}$ be as in Lemma~\ref{lemmaa}. Let also $k,j \in \Z$ and $f(t) \in \F_q[t,t^{-1}]$. Define 
\[ 
\xymatrix@R=2mm{ \phee_{B} : B_2^+(\F_q[t,t^{-1}]) \ar[r] & B_2^+(\F_q[t,t^{-1}]) \\ 
{\begin{pmatrix} t^k & f(t) \\ 0 & t^j \end{pmatrix}} \ar@{|->}[r] & {\begin{pmatrix} t^{-k} & af(t^{-1}) \\ 0 & t^{-j} \end{pmatrix}}} 
\]
and 
\[ 
\xymatrix@R=2mm{ \phee_{A} : \PBplus_2(\F_q[t,t^{-1}]) \ar[r] & \PBplus_2(\F_q[t,t^{-1}]) \\ 
{\begin{bmatrix} t^k & f(t) \\ 0 & t^j \end{bmatrix}} \ar@{|->}[r] & {\begin{bmatrix} t^{-k} & af(t^{-1}) \\ 0 & t^{-j} \end{bmatrix}}.} 
\]
It is a simple exercise to check that these maps are automorphisms. As $\mbU_2(R)$ is characteristic in $B_2^+(R)$ and $\PBplus_2(R)$, we also consider the automorphism  $\phee'= \phee_{B}|_{\mbU_2(R)}=\phee_{A}|_{\mbU_2(R)}$ on $\mbU_2(R)$.

\begin{pps}\label{pps:norifqttmo} The automorphisms defined above satisfy 
\begin{align*}\mathcal{R}(\phee_{B})&=\left\{\left[\left(\begin{smallmatrix} 1&0\\0&1 \end{smallmatrix}\right)\right]_{\phee_{B}},\left[\left(\begin{smallmatrix} t&0\\0&1 \end{smallmatrix}\right)\right]_{\phee_{B}},\left[\left(\begin{smallmatrix} 1&0\\0&t \end{smallmatrix}\right)\right]_{\phee_{B}},\left[\left(\begin{smallmatrix} t&0\\0&t \end{smallmatrix}\right)\right]_{\phee_{B}}\right\},\\
\mathcal{R}(\phee_{A})&=\left\{\left[\left(\begin{smallmatrix} 1&0\\0&1 \end{smallmatrix}\right)\right]_{\phee_{A}},\left[\left(\begin{smallmatrix} t&0\\0&1 \end{smallmatrix}\right)\right]_{\phee_{A}}\right\} \text{ and }\\
\mathcal{R}(\phee')&=\left\{\left[\left(\begin{smallmatrix} 1&0\\0&1 \end{smallmatrix}\right)\right]_{\phee'}\right\},
\end{align*}
where $\mathcal{R}(\phi)$ denotes the set of Reidemeister classes of the automorphism $\phi$.
\end{pps}

\begin{proof} We show the result for $B_{2}^{+}(R)$, the proofs for $\PBplus_2(R)$ and for $\mbU_2(R)$ follow from entirely analogous arguments. 

Given elements of $B_{2}^{+}(R)$ of the form 
\[\left(\begin{smallmatrix} t^{k}&f(t)\\0&t^{\ell}\end{smallmatrix}\right)\text{ and }\left(\begin{smallmatrix} t^{x}&0\phantom{|}\\0&t^{y}\end{smallmatrix}\right),\] 
we have 
\[ \left(\begin{smallmatrix} t^{k}&f(t)\\0&t^{\ell}\end{smallmatrix}\right)\left(\begin{smallmatrix} t^{x}&0\phantom{|}\\0&t^{y}\end{smallmatrix}\right)\phee_{B}\left(\left(\begin{smallmatrix} t^{k}&f(t)\\0&t^{\ell}\end{smallmatrix}\right)\right)^{-1} = \left(\begin{smallmatrix} t^{2k+x}&t^{\ell+y}f(t)-at^{2k+\ell+x}f(t^{-1})\\0&t^{2\ell+y}\end{smallmatrix}\right).\]

Consequently, an element $b=\left(\begin{smallmatrix} t^{i}&h(t)\\0&t^{j}\end{smallmatrix}\right) \in B_{2}^{+}(R)$ 
belongs to the Reidemeister class $\left[\left(\begin{smallmatrix} t^{x}&0\\0&t^{y}\end{smallmatrix}\right)\right]_{\phee_B}$ if and only if the following hold:
\begin{enumerate}
    \item There exist $k,\ell \in \Z$ such that $i = 2k+x$ and $j = 2\ell+y$, and
    \item there exists $f(t) \in \F_q[t,t^{-1}]$ such that 
    \begin{equation}
\label{eq:hf}h(t)=t^{\ell+y}f(t)-at^{2k+\ell+x}f(t^{-1}).
\end{equation}
\end{enumerate}
We now show that such an $f(t)$ always exists for any arbitrary (but fixed) $x,y,k,\ell \in \Z$ and $h(t) \in \F_q[t,t^{-1}]$.

Write $h(t)=\sum_{m\in\Z}h_m t^m$, where only finitely many $h_m\in\F_q$ are non-zero. For each $m \in \Z$, consider the number $\lambda_m=2k+2\ell+x+y-m \in \Z$. We observe that there can be at most one $m \in \Z$ for which the equality $m = \lambda_m$ holds. Now start by defining $f_m$ in the following cases. 
\[f_m=\begin{cases} (1-a)^{-1}h_m, &\text{ if } m=\lambda_m, \\ 
0, &\text{ if } h_m=h_{\lambda_m}=0.
\end{cases}\]
In case it simultaneously holds that $m \neq \lambda_m$ {and} that $h_m$ and $h_{\lambda_m}$ are not {both} zero, we define each pair $(f_m,f_{\lambda_m}) \in \F_q^2$ to be a solution to the system 
\begin{align*} X-aY&=h_m\\
			  -aX+Y&=h_{\lambda_m}
\end{align*}
as in Lemma~\ref{lemmaa}. We remark that, since only finitely many of the coefficients $h_m \in \F_q$ are non-zero, there are only finitely many equations as above to be considered. 

By construction only finitely many coefficients $f_m$ are non-zero, i.e., $f(t)=\sum_{m \in \Z}f_mt^m$ is indeed an element of $\F_q[t,t^{-1}]$. Moreover,  
\begin{align*} t^{\ell+y}f(t)-at^{2k+\ell+x}f(t^{-1}) & = \sum_{m\in\Z} \left(f_mt^{m+\ell+y}-af_mt^{2k+\ell+x-m}\right)\\
& = \sum_{m\in\Z} (f_{m-\ell-y}-af_{2k+\ell+x-m})t^m\\
& = \sum_{m\in\Z} h_mt^m = h(t),
\end{align*}
hence equality~\bref{eq:hf} is satisfied for this choice of $f(t)$. 

The fact that such an $f(t) \in \F_q[t,t^{-1}]$ can always be constructed means that an arbitrary element $b=\left(\begin{smallmatrix} t^{i}&h(t)\\0&t^{j}\end{smallmatrix}\right) \in B_{2}^{+}(R)$ is contained in the Reidemeister class $\left[\left(\begin{smallmatrix} t^{x}&0\\0&t^{y}\end{smallmatrix}\right)\right]_{\phee_B}$ if and only if 
$i \equiv x \mod 2$ and $j \equiv y \mod 2$, without any further assumptions on $h(t)$. Therefore, 
\[\mathcal{R}(\phee_{B})=\left\{\left[\left(\begin{smallmatrix} 1&0\\0&1 \end{smallmatrix}\right)\right]_{\phee},\left[\left(\begin{smallmatrix} t&0\\0&1 \end{smallmatrix}\right)\right]_{\phee},\left[\left(\begin{smallmatrix} 1&0\\0&t \end{smallmatrix}\right)\right]_{\phee},\left[\left(\begin{smallmatrix} t&0\\0&t \end{smallmatrix}\right)\right]_{\phee}\right\},\]
as claimed.
\end{proof}

\section{Proof of Theorem \ref{thm:Lieapplication}, and final remarks} \label{sec:backtoarithmetic}

The promised theorem on $S$-arithmetic groups is now an obvious consequence of Propositions~\ref{pps:FptFptt-1OK} and~\ref{nori} together with known results from the literature.

\begin{proofof}{Theorem \ref{thm:Lieapplication}}
We define the ground field $\K_{p}$ (or $\K_{2,p}$) and the $S$-arithmetic groups $\Gamma_{n,p}$ (or $\til{\Gamma}_{n,p}$) case-by-case. 

If the chosen characteristic $p$ equals zero, set $\K_{0} = \Q(i)$, whose ring of integers is $\Z[i]$. 
For the arithmetic groups, take 
\[\Gamma_{2,0} = \mbU_2(\Z[i]) \rtimes_J C_3 \quad \text{ and } \quad  \Gamma_{n,0} = \PB_n(\Z[i]) \text{ for } n \geq 3,\]
where the generator of the cyclic part of $\Gamma_{2,0}$ acts on $\mbU_2(\Z[i]) \cong \Z^2$ via the matrix $J = \left( \begin{smallmatrix} -1 & 1 \\ -1 & 0 \end{smallmatrix} \right)$. The groups $\Gamma_{n,0} \leq \PB_n(\Q(i))$ with $n \geq 3$ are obviously arithmetic subgroups of $\PB_n(\Q(i))$, while $\Gamma_{2,0}$ is an arithmetic subgroup of $\PB_2(\Q(i))$ since it is commensurable with $\PB_2(\Z[i])$ via their common subgroup of finite index $\mbU_2(\Z[i])$. By Proposition~\ref{pps:FptFptt-1OK}, $\Gamma_{n,0}$ has property~$\Ri$ whenever $n \geq 4$. To see that $\{\infty\}\neq \RSpec(\Gamma_{2,0}) \neq \{1,\infty\}$ we draw from work of Dekimpe--Pennickx. More precisely, by (the proof of) \cite[Theorem~4.1]{KarelPenni}, the crystallographic group $\Gamma_{2,0} \cong \Z^2 \rtimes_J C_3$ admits an automorphism $\phee_D$ with finite Reidemeister number given by conjugation by the matrix $D = \left( \begin{smallmatrix} 0 & 1 \\ -1 & 1 \end{smallmatrix} \right)$. Since $D$ and $J$ commute, the map $\phee_D$ leaves the quotient $C_3$ invariant and thus $\phee_D$ has at least three twisted conjugacy classes. Thus $\RSpec(\Gamma_{2,0})$ is neither $\{\infty\}$ nor $\{1,\infty\}$. (The fact that all $\Gamma_{n,0}$ are finitely presented is well-known, but can also be seen directly from the fact that $\Gamma_{n,0}$ is virtually polycyclic.)

For the case in which the chosen characteristic $p$ is positive, we can set $\K_{p} = \F_p(t)$ for all $n$. Then $\Gamma_{n,p} = \PB_n(\F_p[t])$ gives, by Propositions~\ref{nori}\bref{nori.i} and~\ref{pps:FptFptt-1OK}, the required properties for the Reidemeister spectra. 
However, this choice yields groups $\PB_n(\F_p[t])$ that are \emph{non}-finitely generated, which is a well-known statement; see, e.g., \cite[Theorem~1.3]{YuriSoluble}. 

To find finitely generated $S$-arithmetic candidates satisfying the required properties in positive characteristic, we split into the cases $p \in \{2,3\}$ and $p \geq 5$. 
If the chosen characteristic $p$ is at least $5$, we take $\K_{p} = \F_p(t)$ and set 
\[ \Gamma_{2,p} = \PBplus_2(\F_p[t,t^{-1}]) \quad \text{ and } \quad \Gamma_{n,p} = \PB_n(\F_p[t,t^{-1}]) \text{ for } n \geq 3, \] 
noting that $\Gamma_{2,p}$ is commensurable with $\PB_2(\F_p[t,t^{-1}])$, and apply Propositions~\ref{nori}\bref{nori.ii} and~\ref{pps:FptFptt-1OK}.  
For $p \in \{2,3\}$ one needs to slightly vary the fields. In this case, set 
\[ \K_{2,p} = \F_{p^2}(t) \quad \text{ and } \quad \K_{p} = \F_p(t), \]
and take 
\[ \til{\Gamma}_{2,p} = \PBplus_2(\F_{p^2}[t,t^{-1}]) \quad \text{ and } \quad \til{\Gamma}_{n,p} = \PB_n(\F_p[t,t^{-1}]) \text{ for } n \geq 3. \] 
These groups are all finitely generated; see, e.g., \cite[Theorem~1.3]{YuriSoluble}. 
The proof concludes as in the previous paragraph. 
%
%
%
\end{proofof}

The proof of Theorem~\ref{thm:Lieapplication} relied on the flexibility of the definition of $S$-arithmetic groups, which allows one to pass over to suitable commensurable groups. We observe, for instance, that the arithmetic subgroup $\PB_2(\Z[i]) \leq \PB_2(\Q(i))$ {does} have property~$\Ri$, as can be seen from the classification of low-dimensional crystallographic groups with~$\Ri$; cf. \cite{KarelPenni}. And as discussed in Section~\ref{sec:newexampleswithoutRinfty} it is unknown, for example, whether $\PB_2(\F_5[t,t^{-1}])$ has property~$\Ri$, whence the descent to the subgroup of finite index $\PBplus_2(\F_5[t,t^{-1}])$.

On a similar token, to get a class-$2$ soluble, non-nilpotent, finitely generated $S$-arithmetic group $\Gamma$ with $\{\infty\} \neq \RSpec(\Gamma) \neq \{1,\infty\}$ and in characteristic two or three, we also needed to pass over to extensions of $\F_2(t)$ and $\F_3(t)$ since the groups $\PBplus_2(\F_2[t,t^{-1}]) \cong C_2 \wr \Z$ and $\PBplus_2(\F_3[t,t^{-1}]) \cong C_3 \wr \Z$ have~$\Ri$ as shown in~\cite{DacibergWongWreath}.



\section*{Acknowledgments}
PMLA was supported by the long term structural funding \emph{Methusalem grant} of the Flemish Government. YSR was partially supported by the German Research Foundation (DFG) through the Priority Program 2026 `Geometry at infinity'. 
Part of this project was inspired by a talk given by Peter Wong at the ICM 2018 Satellite Conference ``Geometric Group Theory'' in Campinas, Brazil. We thank him and Timur Nasybullov for helpful discussions. We are also indebted to Karel Dekimpe for valuable remarks, particularly for pointing us to a strengthening of Theorem~\ref{thm:Additive} and some corrections throughout an earlier draft. The authors thank the anonymous referee for a very careful reading of earlier versions of the paper and for the many suggestions for improvement and corrections.

 \def\cprime{$'$} \def\cprime{$'$}
 \providecommand{\bysame}{\leavevmode\hbox to3em{\hrulefill}\thinspace}
 \providecommand{\MR}{\relax\ifhmode\unskip\space\fi MR }

\printbibliography

\end{document}